\documentclass{article}

\usepackage[a4paper,top=2.54cm,bottom=2.54cm,left=3.17cm,right=3.17cm,%
includehead,includefoot]{geometry}

\usepackage{amsmath,amssymb,amsfonts,amsthm}
\usepackage{graphicx}
\allowdisplaybreaks
\usepackage{subfigure}
\usepackage{float}
\usepackage{fmtcount}
\usepackage{xcolor}
\usepackage[numbers,square,sort&compress]{natbib}
\usepackage{hyperref}
\hypersetup{colorlinks,citecolor=blue,linkcolor=blue,breaklinks=true}
\usepackage{booktabs}
\usepackage{colortbl}
\usepackage{caption}
\usepackage{enumitem}
\usepackage{epstopdf}
\usepackage{algorithm}
\usepackage{algpseudocode}
\usepackage{array}
\usepackage{bbding}
\usepackage{fancyhdr} 
\usepackage{fancyvrb} 
\usepackage{longtable}
\usepackage{listings}
\usepackage{rotating,rotfloat} 
\usepackage{yhmath} 
\usepackage{amsthm,amsmath,amssymb}
\usepackage{mathrsfs}
\usepackage{amssymb}
\usepackage{indentfirst}
\usepackage[all]{xy}

\newtheorem{theorem}{Theorem}[section]
\newtheorem{lemma}{Lemma}[section]
\newtheorem{corollary}[theorem]{Corollary}
\newtheorem{proposition}[theorem]{Proposition}
\newtheorem{definition}{Definition}[section]

\newtheorem{conj}[theorem]{Conjecture}

\begin{document}
	\title{\uppercase{Uniform vector bundles over $\mathbb{P}^4$}}
\author{Rong Du \thanks{School of Mathematical Sciences, Key Laboratory of MEA (Ministry of Education) and
Shanghai Key Laboratory of PMMP,
East China Normal University,
Rm. 312, Math. Bldg, No. 500, Dongchuan Road,
Shanghai, 200241, P. R. China,
rdu@math.ecnu.edu.cn.
}
and Yuhang Zhou
\thanks{School of Mathematical Sciences,
University of Science and Technology of China,
No.96, JinZhai Road Baohe District,
Hefei, Anhui, 230026, P. R. China.
zyh@ustc.edu.cn
Both authors are sponsored by National Natural Science Foundation of China (Grant No. 12471040), Innovation Action Plan (Basic research projects) of Science and Technology Commission of Shanghai Municipality (Grant No. 21JC1401900), Natural Science Foundation of Chongqing Municipality, China (general program, Grant No. CSTB2023NSCQ-MSX0334) and Science and Technology Commission of Shanghai Municipality (Grant No. 22DZ2229014).
}
}
	
	\date{}
	
	\maketitle

	\begin{abstract}
		There is a long-standing conjecture which states that every uniform algebraic vector bundle of rank $r<2n$ on the $n$-dimensional projective space $\mathbb{P}^n$ over an algebraically closed field of characteristic $0$ is homogeneous. This conjecture is valid for $n\leq3$. In this paper, we classify all uniform vector bundles of rank $r<8$ over  $\mathbb{P}^4$ and show that the conjecture holds for $n=4$.
	\end{abstract}

	\section{Introduction}
Algebraic vector bundles on a projective variety $X$ are fundamental research objects in algebraic geometry. By Grothendieck's well-known result, if $X=\mathbb{P}^1$, then any vector bundle over $X$ splits as a direct sum of line bundles. However, when $X$ is a projective space of dimension greater than one, determining the structures of vector bundles on $X$ is not as straightforward. Since projective spaces are covered by lines, it is a natural approach to consider the restriction of vector bundles to the lines within them. If the splitting type of a vector bundle $E$ remains the same when it is restricted to any line in the projective space $\mathbb{P}^n$, then $E$ is called a uniform vector bundle on $\mathbb{P}^n$.

Let $k$ be an algebraically closed field of characteristic $0$. All objects are considered over the field $k$, and in the following context, we simply use ``vector bundles'' instead of ``algebraic vector bundles'' and use $r$-bundle instead of ``vector bundle of rank $r$''.

In 1972, Van de Ven (\cite{Ven}) proved that for $n > 2$, uniform $2$-bundles over $\mathbb{P}^n$ split, and uniform $2$-bundles over $\mathbb{P}^2$ are precisely the bundles $\mathcal{O}_{\mathbb{P}^2}(a)\oplus\mathcal{O}_{\mathbb{P}^2}(b)$, $T_{\mathbb{P}^2}(a)$, and $\Omega^1_{\mathbb{P}^2}(b)$, where $a, b\in\mathbb{Z}$.
In 1976, Sato \cite{Sat} proved that for $2 < r < n$, uniform $r$-bundles over $\mathbb{P}^n$ split. In 1978, Elencwajg \cite{Ele} extended Van de Ven's investigations and showed that, up to dual, uniform vector bundles of rank 3 over $\mathbb{P}^2$ are of the forms
$$\mathcal{O}_{\mathbb{P}^2}(a)\oplus\mathcal{O}_{\mathbb{P}^2}(b)\oplus\mathcal{O}_{\mathbb{P}^2}(c), \ T_{\mathbb{P}^2}(a)\oplus\mathcal{O}_{\mathbb{P}^2}(b)\quad\text{and}\quad S^2T_{\mathbb{P}^2}(a),$$
where $a$, $b$, $c\in \mathbb{Z}$. Previously, Sato \cite{Sat} had shown that for odd $n$, uniform $n$-bundles over $\mathbb{P}^n$ are of the forms
$$\bigoplus_{i = 1}^{n}\mathcal{O}_{\mathbb{P}^n}(a_i), \ T_{\mathbb{P}^n}(a)\quad\text{and}\quad \Omega^1_{\mathbb{P}^n}(b),$$
where $a_i,a, b\in \mathbb{Z}$.
The results of Elencwajg and Sato thus provide a complete classification of uniform $3$-bundles over $\mathbb{P}_k^n$. In particular, all uniform $3$-bundles over $\mathbb{P}_k^n$ are homogeneous, meaning they retain their structure under the automorphism pullback of the projective spaces.
Later, Elencwajg, Hirschowitz, and Schneider \cite{E-H-S} showed that Sato's result also holds for even $n$. Around 1982, Ellia \cite{Ell} proved that for $n + 1=4,5,6$, uniform $(n + 1)$-bundles over $\mathbb{P}^{n}$ are of the form
$$\bigoplus_{i = 1}^{n + 1}\mathcal{O}_{\mathbb{P}^n}(a_i), \ T_{\mathbb{P}^n}(a)\oplus\mathcal{O}_{\mathbb{P}^n}(b)\quad\text{and}\quad\Omega^1_{\mathbb{P}^n}(c)\oplus\mathcal{O}_{\mathbb{P}^n}(d), $$
where $a_i, a,b,c,d\in\mathbb{Z}$. Later, Ballico \cite{Bal} showed that Ellia's result remains valid for any $n$. Therefore, uniform $(n + 1)$-bundles over $\mathbb{P}^{n}$ are homogeneous.

One can observe from Ellia and Ballico's papers that the classification problem becomes more difficult when the rank of a vector bundle is greater than $n$.

It is obvious that homogeneous vector bundles are also uniform. Initially, based on the classification results of uniform $2$-bundles and $3$-bundles over projective spaces (\cite{Ven}, \cite{Sat}, and \cite{Ele}), there was a widespread belief that all uniform bundles are homogeneous. However, in 1979, Elencwajg disproved this belief by constructing a non-homogeneous uniform $4$-bundle over $\mathbb{P}^2$ (\cite{Ele2}). Subsequently, Hirschowitz provided examples of uniform non-homogeneous bundles of rank $3n- 1$ over $\mathbb{P}^n$ for $n\geq 3$ (\cite{O-S-S}). In 1980, Dr{\'e}zet proved the existence of uniform non-homogeneous bundles of rank $2n$ over $\mathbb{P}^n$ for $n\geq 3$ (\cite{Dre}).

Then, according to the classification results of uniform vector bundles of rank less than or equal to $n+1$ on $\mathbb{P}^n$, Dr{\'e}zet conjectured in his paper (\cite{Dre}) that uniform vector bundles of rank greater than $n+1$ and less than $2n$ on $\mathbb{P}^n$ are homogeneous (or see \cite{EP-MP}).

\begin{conj}
    Uniform vector bundles of rank $r$ such that $n + 1<r<2n$ over $\mathbb{P}^n$ are homogeneous.
\end{conj}

In 1983, Ballico and Ellia (\cite{B-E-P3}) demonstrated that the conjecture holds true for $n=3$. However, to the best of our knowledge, over the past four decades, despite the concerted efforts of the research community, there has been a dearth of significant progress in higher-dimensional cases. As the value of $n$ increases, the complexity of the problem appears to grow exponentially, presenting formidable challenges that have hitherto proven difficult to surmount.

In this paper, we attempt to contribute to this long-standing problem. We aim to provide a positive solution for the case of $n = 4$, fully aware that this is but a small step in the face of the overwhelming difficulties that come with higher-dimensional scenarios. We hope that our work might inspire further research and help pave the way for more comprehensive solutions as we continue to grapple with the increasing complexity associated with larger values of $n$.

\begin{theorem}
    Uniform $r$-bundles with $r < 8$ on $\mathbb{P}^4$ are homogeneous.
\end{theorem}

\section{Preliminaries}
	
In this section, we present some basic concepts and theorems that will be used. General background information can be found in \cite{E-H-S}.
	
\subsection{Chern polynomial}
		
	\begin{definition}
		The Chern polynomial of a vector bundle $E$ of rank r over variety X is
		\begin{center}
			$c_{E}(T)=T^r-c_1(E)T^{r-1}+c_2(E)T^{r-2}+\cdots+(-1)^{r}c_{r}(E)$.
		\end{center}
	\end{definition}
	\begin{theorem}
		For short exact sequence of vector bundles over $X$
		\begin{center}
			$0\longrightarrow F\longrightarrow G\longrightarrow E\longrightarrow 0$,
		\end{center}
we have	$c_{G}(T)\equiv c_{E}(T) c_{F}(T) $ mod $A(X)$, where $A(X)$ is the Chow ring of $X$.
	\end{theorem}
	
	\subsection{Uniform vector bundles}
	\begin{theorem}
	(Grothendieck 1956) For every vector bundle	$E$ over $\mathbb{P}^1$, $E\cong \underset{i=1}{\overset{k}{\bigoplus}}\mathcal{O}_{\mathbb{P}^1}^{\oplus r_i}(u_i)$.
	\end{theorem}
\begin{definition}
	A vector bundle $E$ over $\mathbb{P}^n$ is uniform if for every line $L\subset \mathbb{P}^n$,  $$E|_{L}\cong \underset{i=1}{\overset{k}{\bigoplus}}\mathcal{O}_L^{\oplus r_i}(u_i),$$ where $u_1>\cdots >u_k$.
	
	The splitting type of the uniform vector bundle $E$ is $(k;r_1,\cdots,r_k;u_1,\cdots, u_k)$.
\end{definition}

	 \begin{theorem}
	 (\cite{Ell},Proposition 1-4.1)	The uniform vector bundle $E$ is an extension of two uniform vector bundles when $u_i-u_{i+1}\geq 2$ .
	 \end{theorem}
	
	 \subsection{HN-filtration}
	
There is a standard commutative diagram :
	\begin{align}
		\xymatrix
		{
			F(1,2,n+1)\ar[r]^{q} \ar[d]^{p}& F(2,n+1)\\
			\mathbb{P}^n,  & \\
		}
	\end{align}
	where $F(1,2,n+1)$ and $F(2,n+1)=\mathbb{G}(1,n)$ are flag variety. The morphisms in this diagram are all projective. As a result,  $F(1,2,n+1)=\mathbb{P}(T_{\mathbb{P}^n}(-1))$. For the sake of brevity in the context, we further define  $F_n:=F(1,2,n+1)$ in the context for shorten the notation. Denote the Hopf relative bundle by $\mathcal{H}_{T_{\mathbb{P}^n}(-1)}:=\mathcal{O}_{\mathbb{P}(T_{\mathbb{P}^n}(-1))}(-1)$.

	For any uniform vector bundle $E$ over $\mathbb{P}^n$ with splitting type $(k;r_1,\cdots,r_k;u_1,\cdots,u_k)$, there is a HN-filtration of $p^{*}E$:
	\begin{center}
		$p^{*}E=HN^k\supset HN^{k-1}\supset\cdots \supset HN^{1}\supset HN^{0}=0$.
	\end{center}
Moreover, there is a short exact sequence $$0\longrightarrow HN^{i-1}\longrightarrow HN^{i}\longrightarrow q^{*}E_{i}\otimes p^*\mathcal{O}_{\mathbb{P}^n}(u_i)\longrightarrow 0,$$
where $E_i$ is a vector bundle over $F(2,n+1)$ for $0<i\leq k$.

Thus,
	\begin{center}
		$c_{p^{*}E}(T,U)=\underset{i=1}{\overset{k}{\prod}}S_{i}(T+u_iU,U,V)$ in  $A(F(1,2,n+1)),$
	\end{center}

where $U=c_1(p^{*}\mathcal{O}_{\mathbb{P}^n}(-1))$, $V=c_{1}(\mathcal{H}_{T_{\mathbb{P}^n}(-1)})$,
$c_{p^{*}E}(T,U)$ is the Chern polynomial of $p^{*}E$,  $S_{i}(T,U,V)$ which is symmetric about $U$ and $V$ is the Chern polynomial of $q^*E_{i}$ and
 $A(F(1,2,n+1))$ is the Chow ring of $F(1,2,n+1)$. See \cite{E-H-S} for more details.

The following theorem is about the Chow group of $F(1,2,n+1)$.
	\begin{theorem} (\cite{Guyot}Theorem 3.1.)\label{Guy}
		$A(F(1,2,n+1))\cong \mathbb{Z}[U,V]/<R^{n}(U,V),U^{n+1}>$, where $R^{n}(U,V)=
		\underset{k=0}{\overset{n}{\sum}}U^kV^{n-k}$.
	\end{theorem}
	
	\begin{lemma}\label{deco-lem}
		For a point $x\in \mathbb{P}^n$, if $HN^{i}|_{p^{-1}(x)}\cong \mathcal{O}_{p^{-1}(x)}\oplus G$, where $G$ is a vector bundle over $\mathbb{P}^{n-1}$ with $H^{0}(p^{-1}(x), G)=H^{1}(p^{-1}(x), G)=0$, then $E$ has $\mathcal{O}_{\mathbb{P}^n}(l)$ as its subbundle, where $l$ is determined by the Chern polynomial of $HN^i$.
	\end{lemma}
\begin{proof}
	
Given that $H^{0}(p^{-1}(x), HN^{1}|_{p^{-1}(x)})=1$ and $H^{1}(p^{-1}(x), HN^{1}|_{p^{-1}(x)})=0$, we can conclude that there exists an integer $l\in \mathbb{Z}$ such that  $p_{*}HN^i\cong \mathcal{O}_{\mathbb{P}^n}(l)$ and $R^{1}p_{*}HN^1=0$.
	
Now, consider the following exact sequence:
	\[0\longrightarrow HN^i\longrightarrow p^{*}E\longrightarrow F^i\longrightarrow 0,\] where $F^i$ is a vector bundle. By applying the push forward functor $p_*$ to this sequence, we obtain the following exact sequence:
\[0\longrightarrow  \mathcal{O}_{\mathbb{P}^n}(l)\longrightarrow E\longrightarrow p_{*}F^{i}\longrightarrow0. \]
Furthermore, when we apply the pull-back functor  $p^{*}$ to the above exact sequence, we get the following commutative diagram:
	\begin{center}
		$\xymatrix{
	0\ar[r]& p^{*}	\mathcal{O}_{\mathbb{P}^n}(l)\ar[r]\ar[d]&p^{*}E\ar[r] \ar[d]^{id}& p^{*}p_{*}F^i\ar[r]\ar[d]&0\\
	0\ar[r]& HN^i \ar[r] & p^{*}E \ar[r]& F^i \ar[r]&0.\\
	}$
	\end{center}
By the snake lemma, we can infer that $HN^i$ has $p^{*}\mathcal{O}_{\mathbb{P}^n}(l)$ as a subbundle.  So $E$ has $\mathcal{O}_{\mathbb{P}^n}(l)$ as a subbundle.
	
\end{proof}

\begin{lemma}\label{trivial-lemma}
	If $HN^{i}/HN^{i-1}$ is trivial on all the $p$-fibers, then $p_{*}(HN^{i}/HN^{i-1})\otimes\mathcal{O}_{\mathbb{P}^n}(-u_i)$ is trivial bundle of rank $r_i$ over $\mathbb{P}^n$ and $HN^{i}/HN^{i-1}\cong p^*\mathcal{O}^{\oplus r_i}_{\mathbb{P}^n}(u_i)$.
\end{lemma}

\begin{proof}
    Since $HN^{i}/HN^{i - 1}$ is trivial on the $p$-fiber, we know that $p_{*}(HN^{i}/HN^{i - 1})=P$ is a vector bundle over $\mathbb{P}^n$. It is evident that the canonical morphism from $p^{*}P$ to $HN^{i}/HN^{i - 1}$ is an isomorphism.

    Meanwhile, according to Proposition 3.3 in \cite{E-H-S}, there exists a vector bundle $E_{i}$ over $\mathbb{G}(1, n)$ such that $HN^{i}/HN^{i - 1}\cong q^{*}E_{i}\otimes p^*\mathcal{O}_{\mathbb{P}^n}(u_i)$. From this, we can conclude that both $P\otimes \mathcal{O}_{\mathbb{P}^n}(-u_i)$ and $E_{i}$ are trivial.
\end{proof}


\begin{corollary}\label{all-trivial-lem}
	Let $E$ be a uniform vector bundle of splitting type $(k;u_1,\cdots,u_k;r_1,\cdots,r_k)$. If $HN^i/HN^{i-1}$ are trivial over the $p$-fibers for $1\leq i \leq k$, then $E$ is direct sum of line bundles.
\end{corollary}
\begin{proof}
	By above lemma and induction, the $HN$-filtration is given by
	\[HN^1\cong p^{*}\mathcal{O}^{\oplus r_{1}}_{\mathbb{P}^n}(u_{1})\]
		
and \[0\longrightarrow HN^i\longrightarrow HN^{i+1}\longrightarrow p^{*}\mathcal{O}^{\oplus r_{i+1}}_{\mathbb{P}^n}(u_{i+1})\longrightarrow 0\] for $i=1,\dots, k-1$.

So $HN^i\cong \underset{j=1}{\overset{i}{\bigoplus}} p^{*}\mathcal{O}^{\oplus r_{j}}_{\mathbb{P}^n}(u_{j})$,
and $E=p_{*}p^{*}E=p_{*}(\underset{j=1}{\overset{k}{\bigoplus}} p^{*}\mathcal{O}^{\oplus r_{j}}_{\mathbb{P}^n}(u_{j}))=\underset{j=1}{\overset{k}{\bigoplus}} \mathcal{O}^{\oplus r_{j}}_{\mathbb{P}^n}(u_{j})$.
\end{proof}

If $k = 1$, it is a well-established fact that the uniform bundle $E$ can be expressed as the direct sum of identical line bundles.
When $k = 2$, Ellia managed to prove the following result.

%
%
\begin{lemma}\label{k=2r1=1}
	(\cite{Ell} Proposition  \uppercase\expandafter{\romannumeral4}-2.2) Any uniform $r$-bundle $E$ with splitting form $(2;1,r-1,u_1,u_1-1)$ over $\mathbb{P}^n$ is homogeneous.
\end{lemma}

So by Ellia's result, we only need to consider the cases of $k>2$ and $k=2$ with $r_1>1$.

\section{Uniform $6$-bundles over $\mathbb{P}^4$}

	\subsection{$k=2,r_1=r_2=3$}
	
Now
	 \begin{equation}\label{r1r2=3}
c_{p^{*}E}(T)=\overset{2}{\underset{i=1}{\prod}}S_i(T+u_{i}U,U,V)
\end{equation}
in $A(F(1,2,5))$, so we have
\[T^{2}P(T,U)+A(T,U,V)R^{4}(U,V)+B(T,U,V)R^{5}(U,V)=\overset{2}{\underset{i=1}{\prod}}S_i(T+u_{i}U,U,V),\]	
where $P(T,U)$, $A(T,U,V)$ and $B(T,U,V)$ are homogeneous polynomials of degrees $4$, $2$ and $1$ respectively.
Assume that
\[P(T,U)=T^4+c_{1}UT^3+c_{2}U^2T^2+c_{3}U^3T+c_{4}U^4,\]
		
		\[A(T,U,V)=a_{1}T^2+(a_{2}U+a_{3}V)T+A_{2}(U,V),\]
		
	and \[B(T,U,V)=b_{1}T+b_{2}U+b_{3}V,\] where $A_{2}(U,V)$ is of degree $2$, $a_{1}, a_{2}, a_{3}, b_{1}, b_{2}, b_{3}$ and $c_{1}\dots c_{4} $ are all constants.

Suppose that 	
		$$S_{1}(T,U,V)=T^3+t_{1}(U+V)T^2+[t_{2}^{1}(U^2+V^2)+t_{2}^{2}UV]T+[t_{3}^{1}(U^3+V^3)+t_{3}^{2}(U^{2}V+UV^{2})]$$
and	
	$$S_{2}(T,U,V)=T^3+n_{1}(U+V)T^2+[n_{2}^{1}(U^2+V^2)+n_{2}^{2}UV]T+[n_{3}^{1}(U^3+V^3)+n_{3}^{2}(U^{2}V+UV^{2})],$$
where $t_{1}, t_{2}^{1}, t_{2}^{2}, t_{3}^{1}, t_{3}^{2}, n_{1}, n_{2}^{1}, n_{2}^{2}, n_{3}^{1}, n_{3}^{2}$ are all constants.
	
	Without loss of generality, we may assume $u_1=1$ and $u_2=0$ .
	
	
Using software Mathematica 11.3 (See the Appendix $(3,3)$-case), we get only four solutions of the equation (\ref{r1r2=3}).
\begin{eqnarray}
(t_1,t_2^1,t_2^2,t_3^1,t_3^2,n_1,n_2^1,n_2^2,n_3^1,n_3^2)&=&(0,0,0,0,0,0,0,0,0,0),\nonumber\\
&&(-1,1,1,-1,-1,1,0,0,0,0),\nonumber\\
&&(-1,0,0,0,0,1,1,1,1,1),\nonumber\\
&&(-2,2,3,0,-1,2,2,3,0,1).
\end{eqnarray}


\begin{proposition}\label{(0,0,0)case}
If $(t_1,t_2^1,t_2^2,t_3^1,t_3^2,n_1,n_2^1,n_2^2,n_3^1,n_3^2)=(0,0,0,0,0,0,0,0,0,0),$ then $$E\cong  \mathcal{O}_{\mathbb{P}^{4}}^{\oplus 3}(1)\oplus \mathcal{O}_{\mathbb{P}^{4}}^{\oplus 3}.$$
\end{proposition}
\begin{proof}
The restriction of $HN^1$ to a $p$-fiber at some point $x\in \mathbb{P}^4$, $HN^1|_{p^{-1}(x)}$, is subbundle of the trivial bundle with $c_1=0$. Thus $HN^1|_{p^{-1}(x)}$ is trivial bundle (cf. \cite{O-S-S} Corollary after Theorem 3.2.1).  By Corollary \ref{all-trivial-lem}, $E\cong  \mathcal{O}_{\mathbb{P}^{4}}^{\oplus 3}(1)\oplus \mathcal{O}_{\mathbb{P}^{4}}^{\oplus 3} $.
\end{proof}

\begin{proposition}\label{(-1,0,0)case}
If $(t_1,t_2^1,t_2^2,t_3^1,t_3^2,n_1,n_2^1,n_2^2,n_3^1,n_3^2)=(-1,0,0,0,0,1,1,1,1,1)$, then $$E\cong T_{\mathbb{P}^4}(-1)\oplus \mathcal{O}_{\mathbb{P}^4}^{\oplus 2}(1).$$
\end{proposition}
\begin{proof}
Now $S_1(T,U,V )=T^2(T-(U+V))$, so the Chern polynomial of $HN^1$ is $$c_{HN^1}(T,U,V)=(T+U)^2(T-V)=T^3+(2U-V)T^2+(U^2-2UV)T-U^2V.$$
	By the definition of $HN$-filtration, for a point $x\in \mathbb{P}^4$, $HN^1|_{p^{-1}(x)}$ is subbundle of trivial bundle with $c_1(HN^1|_{p^{-1}(x)})=\mathcal{O}_{\mathbb{P}^3}(-1)$. So $HN^1|_{p^{-1}(x)}\cong \mathcal{O}_{\mathbb{P}^3}(-1)\oplus \mathcal{O}_{\mathbb{P}^3}^{\oplus 2}$.
	Denoting $M=p_{*}HN^1$ which is a rank $2$ vector bundle and applying $p_{*}$ to the exact sequence
	\begin{center}
		$0\longrightarrow HN^1\longrightarrow p^{*}E \longrightarrow F\longrightarrow 0,$
	\end{center}
	we have the exact sequence
	$$0\longrightarrow M \longrightarrow E \longrightarrow p_{*}F\longrightarrow 0,$$ where $F$ is some vector bundle over $F(1,2,5)$.
	Then we have the following commutative diagram
	\begin{center}
		$\xymatrix{
			0\ar[r] &p^{*}M \ar[r] \ar[d] & p^{*}E \ar[d]^{id} \ar[r] & p^{*}p_{*}F\ar[r]\ar[d]&0\\
			0\ar[r] & HN^{1} \ar[r]&p^{*}E \ar[r]&F\ar[r]&0.\\
		}$
	\end{center}
	By the snake lemma, the following exact sequences hold:
\[0\longrightarrow p^{*}M\longrightarrow HN^1\longrightarrow G\longrightarrow 0\]
and 		
		\[0\longrightarrow G\longrightarrow p^{*}p_{*}F\longrightarrow F\longrightarrow 0,\]
	where $G=HN^1/p^{*}M$ is the quotient bundle. Comparing the Chern polynomials of $G, p^*M$ and $HN^1$, we can get $c_{G}(T,U,V)=T-V$ and $c_{p^*M}(T, U)=(T+U)^2$.
Thus $p^{*}p_{*}F$ has $\mathcal{H}_{T_{\mathbb{P}^n}(-1)}$ as its subbundle. By Proposition 3.5 in \cite{E-H-S}, $p_{*}F$ has $T_{\mathbb{P}^4}(-1)$ as its subbundle. So $p_{*}F\cong T_{\mathbb{P}^4}(-1)$ whose splitting type is $(2;1,0;1,3)$. Then $M$ is a uniform vector bundle over $\mathbb{P}^4$.
	Thus, we have the following exact sequence
	\[0\longrightarrow \mathcal{O}_{\mathbb{P}^4}^{\oplus 2}(1)\longrightarrow E\longrightarrow T_{\mathbb{P}^4}(-1)\longrightarrow 0.\]
	Since $Ext^1(T_{\mathbb{P}^4}(-1), \mathcal{O}_{\mathbb{P}^4}^{\oplus 2}(1))=0$, we get that $E\cong T_{\mathbb{P}^4}(-1)\oplus \mathcal{O}_{\mathbb{P}^4}^{\oplus 2}(1) $ .
\end{proof}

 \begin{proposition}
 	If $(t_1,t_2^1,t_2^2,t_3^1,t_3^2,n_1,n_2^1,n_2^2,n_3^1,n_3^2)=(-1,1,1,-1,-1,1,0,0,0,0)$, then $$E\cong \Omega^1_{\mathbb{P}^4}(2)\oplus \mathcal{O}_{\mathbb{P}^4}^{\oplus 2}.$$
 \end{proposition}
\begin{proof}
	By the same argument above, $E^{*}(1)\cong T_{\mathbb{P}^4}(-1)\oplus \mathcal{O}_{\mathbb{P}^4}^{\oplus 2}(1)$. So $E\cong \Omega^1_{\mathbb{P}^4}(2)\oplus \mathcal{O}_{\mathbb{P}^4}^{\oplus 2}$.
\end{proof}

\begin{proposition}\label{exterior T}
	If $(t_1,t_2^1,t_2^2,t_3^1,t_3^2,n_1,n_2^1,n_2^2,n_3^1,n_3^2)=(-2,2,3,0,-1,2,2,3,0,1)$, then
	\[
	E\cong \bigwedge^{2}\bigl(T_{\mathbb{P}^4}(-1)\bigr).
	\]
\end{proposition}
\begin{proof}
Put $X=F(1,2,5)$ and $H=HN^1$.  Since $u_1=1$, the Chern polynomial of $H$ is
\begin{align*}
c_H(T)&=S_1(T+U,U,V)\\
&=T^3+(U-2V)T^2+(U^2-UV+2V^2)T+U^3+UV^2.
\end{align*}
In particular, for every $x\in\mathbb{P}^4$, the bundle $H_x:=H|_{p^{-1}(x)}$ is a rank-three subbundle of
$\mathcal O_{\mathbb P^3}^{\oplus 6}$ with Chern classes
\[
c_1(H_x)=-2h,\qquad c_2(H_x)=2h^2,\qquad c_3(H_x)=0.
\]
The classification of globally generated bundles with first Chern class two in \cite{S-J-U}, applied to $H_x^\vee$, gives
\[
H_x\cong T_{\mathbb P^3}(-2)
\quad\text{or}\quad
H_x\cong N(-1)\oplus\mathcal O_{\mathbb P^3},
\]
where $N$ is a null-correlation bundle.  Their cohomology is
\[
H^i\bigl(T_{\mathbb P^3}(-2)\bigr)=0\quad\text{for all }i,
\]
and
\[
h^0\bigl(N(-1)\oplus\mathcal O\bigr)=
h^1\bigl(N(-1)\oplus\mathcal O\bigr)=1,
\qquad h^2=h^3=0.
\]

Let
\[
J=\{x\in\mathbb P^4\mid h^0(H_x)=1\}.
\]
Semicontinuity shows that $J$ is closed.  Suppose first that $J$ is a proper closed subset.  The inclusion
$H\hookrightarrow p^*E$ induces $p_*H\hookrightarrow E$.  Since $p_*H$ vanishes over the dense open set
$\mathbb P^4\setminus J$, it is a torsion subsheaf of the locally free sheaf $E$; hence $p_*H=0$.
The above fiber cohomology also gives $R^2p_*H=R^3p_*H=0$.

We now compute the remaining derived image.  The Chern character of $H$ is determined by $c_H(T)$, so GRR gives
\[
\operatorname{ch}(p_!H)=p_*\bigl(\operatorname{ch}(H)\operatorname{td}(T_p)\bigr).
\]
For the homogeneous bundle
\[
E_0=\bigwedge^2\bigl(T_{\mathbb P^4}(-1)\bigr)
\]
the corresponding first HN term has the same Chern polynomial and restricts to
$T_{\mathbb P^3}(-2)$ on every $p$-fiber.  It is therefore $p$-acyclic.  Applying GRR to this model yields
$\operatorname{ch}(p_!H)=0$ for every $H$ with the above Chern polynomial.  Since in the present case
$p_!H=-[R^1p_*H]$, we obtain $\operatorname{ch}(R^1p_*H)=0$.  A nonzero coherent sheaf on
$\mathbb P^4$ has a nonzero Hilbert polynomial, so $R^1p_*H=0$.  Cohomology and base change now implies
$H^1(H_x)=0$ for every $x$, and hence $J=\varnothing$.

It remains to exclude $J=\mathbb P^4$.  In that case $h^0(H_x)=1$ for all $x$, so Grauert's theorem shows that
$L:=p_*H$ is a line bundle and base change holds.  The evaluation map
\[
p^*L\longrightarrow H
\]
restricts on each fiber to the unique nonzero section of $N(-1)\oplus\mathcal O$; this section is nowhere vanishing.
Consequently the evaluation map is a line subbundle.  Writing $L\cong\mathcal O_{\mathbb P^4}(i)$, the polynomial
$c_H(T)$ must be divisible by $T+iU$.  Substitution of $T=-iU$ shows that the coefficient of $UV^2$ is
$1-2i$, which cannot vanish for an integer $i$.  This contradiction excludes $J=\mathbb P^4$.

Thus $H_x\cong T_{\mathbb P^3}(-2)$ for every $x$.  If $Q=p^*E/H$, then the fiber sequence
\[
0\longrightarrow T_{\mathbb P^3}(-2)
\longrightarrow\mathcal O_{\mathbb P^3}^{\oplus 6}
\longrightarrow Q_x\longrightarrow0
\]
is the evaluation sequence of $\Omega^1_{\mathbb P^3}(2)$; hence $Q_x\cong\Omega^1_{\mathbb P^3}(2)$.
The characterization theorem of Ballico--Ellia \cite{B-E1982} now applies to this relative HN sequence and gives
\[
E\cong\bigwedge^2\bigl(T_{\mathbb P^4}(-1)\bigr).
\]
\end{proof}
So we prove the following proposition.
\begin{proposition}
	The uniform vector $E$ over $\mathbb{P}^4$ with the splitting form $(2;3,3;u_1,u_1-1)$ is
	\begin{center}
		$\mathcal{O}_{\mathbb{P}^{4}}^{\oplus 3}(u_1)\oplus \mathcal{O}_{\mathbb{P}^{4}}^{\oplus 3}(u_1-1)$, $ T_{\mathbb{P}^4}(u_1-2)\oplus \mathcal{O}_{\mathbb{P}^4}^{\oplus 2}(u_1)$,	$ \Omega^1_{\mathbb{P}^4}(u_1+1)\oplus \mathcal{O}_{\mathbb{P}^4}^{\oplus 2}(u_1-1)$ or $\bigwedge^{2}(T_{\mathbb{P}^4}(-1))\otimes\mathcal{O}_{\mathbb{P}^4}(u_1-1)$.
	\end{center}
\end{proposition}

\vspace{1cm}

\subsection{$k=2,r_1=2,r_2=4$}
Similarly, we suppose that
	 \begin{equation}\label{r1=2r2=4}
c_{p^{*}E}(T)=T^{2}P(T,U)+A(T,U,V)R^{4}(U,V)+B(T,U,V)R^{5}(U,V)=\overset{2}{\underset{i=1}{\prod}}S_i(T+u_{i}U,U,V)
\end{equation}
in $A(F(1,2,5))$,
where $P(T,U)$, $A(T,U,V)$ and $B(T,U,V)$ are homogeneous polynomials of degrees $4$, $2$ and $1$ respectively.

Without loss of generality, we set $u_2=0$.
Suppose that
\[S_{1}(T,U,V)=T^2+m_{1}(U+V)T+(m_{2}^{0}(U^2+V^2)+m_{2}^{1}UV)\] and
\[S_{2}(T,U,V)=T^4+N_{1}(U,V)T^3+N_{2}(U,V)T^2+N_{3}(U,V)T+N_4(U,V),\]
where $N_{i}(U,V)$ is homogeneous polynomials of degrees $i$ for $i=1,2,3,4$.
Using computer calculation, we get only following two solutions of equation (\ref{r1=2r2=4}).

\begin{center}
	$(m_1,m_2^0,m_2^1)=(0,0,0),$
	
	$(m_1,m_2^0,m_2^1)=(-1,0,0).$
\end{center}

If $(m_1,m_2^0,m_2^1)=(0,0,0)$, by the similar arguments in the proof of Proposition \ref{(0,0,0)case}, $E\cong \mathcal{O}_{\mathbb{P}^{4}}^{\oplus 2}(1)\oplus \mathcal{O}_{\mathbb{P}^{4}}^{\oplus 4}.$

\begin{proposition}
	If $(m_1,m_2^0,m_2^1)=(-1,0,0)$, then $$E\cong \mathcal{O}_{\mathbb{P}^4}(1)\oplus\mathcal{O}_{\mathbb{P}^4}\oplus T_{\mathbb{P}^4}(-1).$$
\end{proposition}

\begin{proof}
Clearly, $c_{HN^1}(T)=T^2+(U-V)T-UV=(T+U)(T-V)$. So $c_{1}(HN^1)=V-U$ and $c_{1}(HN^1|_{p^{-1}(x)})=\mathcal{O}_{\mathbb{P}^3}(-1)$, for each point $x \in \mathbb{P}^4$. It is easy to see that $HN^1|_{p^{-1}(x)}$ is a subbundle of the trivial bundle. Thus, $HN^1|_{p^{-1}(x)}\cong \mathcal{O}_{\mathbb{P}^3} \oplus \mathcal{O}_{\mathbb{P}^3}(-1)$, for each point  $x \in \mathbb{P}^4.$ By Lemma \ref{deco-lem}, we have the exact sequence 	
	\[0\longrightarrow \mathcal{O}_{\mathbb{P}^4}(1)\longrightarrow E\longrightarrow E^{'}\longrightarrow 0,\] where $E'$ is the quotient bundle.
	By cohomology calculation, it is not hard to get that $E'$ is a uniform bundle with splitting form $(1,0,0,0,0)$. By the classifying result in \cite{Ell} and Chern polynomial of $E'$, $E'\cong T_{\mathbb{P}^4}(-1)\oplus \mathcal{O}_{\mathbb{P}^4}$. So $E\cong \mathcal{O}_{\mathbb{P}^4}(1)\oplus\mathcal{O}_{\mathbb{P}^4}\oplus T_{\mathbb{P}^4}(-1).$
\end{proof}

 So we prove the following proposition:
   \begin{proposition}
   	The uniform vector bundle $E$ over $\mathbb{P}^4$ with splitting form of $(2;2,4;u_1,u_1-1)$ is isomorphic to either $\mathcal{O}_{\mathbb{P}^{4}}^{\oplus 2}(u_1)\oplus \mathcal{O}_{\mathbb{P}^{4}}^{\oplus 4}(u_1-1)$ or $\mathcal{O}_{\mathbb{P}^4}(u_1)\oplus\mathcal{O}_{\mathbb{P}^4}(u_1-1)\oplus T_{\mathbb{P}^4}(u_1-2)$.
   \end{proposition}


 \subsection{$k=3,r_1=1,r_2=2,r_3=3$}
We suppose that
\begin{equation}\label{r1=1r2=2}
c_{p^{*}E}(T)=\overset{3}{\underset{i=1}{\prod}}S_i(T+u_{i}U,U,V).
\end{equation}
Without loss of generality, we may assume $u_3=0$. We also assume that
\[S_{1}(T,U,V)=T+a(U+V),\]
\[S_{2}(T,U,V)=T^2+b_1(U+V)T+b_2^0U^2+b_2^1UV+b_2^0V^2\] and
\[S_{3}(T,U,V)=T^3+n_1(U+V)T^2+(n_2^0U^2+n_2^1UV+n_2^0V^2)T+(n_3^0U^3+n_3^1U^2V+n_3^1UV^2+n_3^0V^3),\]
where $a, b_1, b_2^0, b_2^1, b_2^2, n_1, n_2^0, n_2^1, n_3^0, n_3^1$ are all constants.

Using computer calculation, we get only following three solutions of equation (\ref{r1=1r2=2}).
  \begin{center}
 	$(a,b_1,b_2^0,b_2^1,n_1,n_2^0,n_2^1,n_3^0,n_3^1)=(0,0,0,0,0,0,0,0,0)$,
 	
 	$(a,b_1,b_2^0,b_2^1,n_1,n_2^0,n_2^1,n_3^0,n_3^1)=(0,-1,0,0,1,1,1, 1, 1)$,
 	
 	$(a,b_1,b_2^0,b_2^1,n_1,n_2^0,n_2^1,n_3^0,n_3^1)=(-2, 0, 0, 0,2, 4, 4, 8, 8)$.
 \end{center}
If $(a,b_1,b_2^0,b_2^1,n_1,n_2^0,n_2^1,n_3^0,n_3^1)=(0,0,0,0,0,0,0,0,0)$, then by the similar arguments in the proof of Proposition \ref{(0,0,0)case}, $E\cong \mathcal{O}_{\mathbb{P}^{4}}(2)\oplus\mathcal{O}_{\mathbb{P}^{4}}^{\oplus 2}(1)\oplus \mathcal{O}_{\mathbb{P}^{4}}^{\oplus 3}.$

 \begin{proposition}\label{0-1}
 	If $(a,b_1,b_2^0,b_2^1,n_1,n_2^0,n_2^1,n_3^0,n_3^1)=(0,-1,0,0,1,1,1, 1, 1)$, then  $$E\cong \mathcal{O}_{\mathbb{P}^4}(2) \oplus \mathcal{O}_{\mathbb{P}^4}(1)\oplus  T_{\mathbb{P}^4}(-1).$$
 \end{proposition}
\begin{proof}
Now $S_{1}(T,U,V)=T$ and $c_{HN^{1}}=T+2U$. So for a point $x\in \mathbb{P}^4$, $HN^{1}|_{p_{-1}(x)}\cong \mathcal{O}_{\mathbb{P}^3}$. By Lemma \ref{trivial-lemma} and the Chern polynomial of $HN^1$, $HN^1\cong p^{*}\mathcal{O}_{\mathbb{P}^4}(2)$. Applying $p_{*}$ to the exact sequence
	\[0\longrightarrow HN^1\longrightarrow p^{*}E\longrightarrow F\longrightarrow 0\]
	for some quotient bundle $F$, we get the exact sequence
	\[0\longrightarrow \mathcal{O}_{\mathbb{P}^4}(2)\longrightarrow E\longrightarrow p_*F\longrightarrow0.\]
	It is not hard to see that $ p_*F$ is a uniform bundle. Thus $ p_*F\cong T_{\mathbb{P}^4}(-1)\oplus \mathcal{O}_{\mathbb{P}^4}(1)$ by the result in \cite{Ell}. Thus, $E\cong \mathcal{O}_{\mathbb{P}^4}(2) \oplus \mathcal{O}_{\mathbb{P}^4}(1)\oplus  T_{\mathbb{P}^4}(-1)$.
\end{proof}

 \begin{proposition}\label{-2 case}
 	 If $(a,b_1,b_2^0,b_2^1,n_1,n_2^0,n_2^1,n_3^0,n_3^1)=(-2, 0, 0, 0,2, 4, 4, 8, 8)$, then there does not exist a uniform vector bundle.
 \end{proposition}

 \begin{proof}
 	In this case,  $S_{1}(T,U,V)=T-2(U+V)$, $S_2(T,U,V)=T^2$, $c_{HN^{1}}=T-2V$ and $c_{HN^{2}}=(T-2V)(T+U)^2$. So $HN^{1}|_{p^{-1}(x)}\cong \mathcal{O}_{\mathbb{P}^3}(-2)$,  $c_{1}(HN^2|_{p^{-1}(x)})=\mathcal{O}_{\mathbb{P}^3}(-2)$, $c_2(HN^2|_{p^{-1}(x)})=0$ and $c_3(HN^2|_{p^{-1}(x)})=0$ for each point $x\in \mathbb{P}^4$. Since $HN^{2}|_{p^{-1}(x)}$ is a subbundle of the trivial bundle, its dual is globally generated. By the main theorem in \cite{S-J-U}, $HN^{2}|_{p^{-1}(x)}\cong \mathcal{O}_{\mathbb{P}^3}(-2)\oplus \mathcal{O}^{\oplus 2}_{\mathbb{P}^3}$.
 	
 	So restricting $HN$-sequence to $p$-fiber, we have the exact sequence
 	\[0\longrightarrow \mathcal{O}_{p^{-1}(x)}(-2)\longrightarrow \mathcal{O}_{p^{-1}(x)}(-2)\oplus \mathcal{O}^{\oplus 2}_{p^{-1}(x)}\longrightarrow HN^{2}/HN^1|_{p^{-1}(x)}\longrightarrow 0.\]
 	
 	Denote $M:=HN^{2}/HN^1|_{p^{-1}(x)}$. Then $M$ is a rank $2$ bundle with $c_1=c_2=0$. So $M\cong {M}^{\vee}$. Clearly, $H^{1}(p^{-1}(x), M(k))=0$ for all $k$. By Horrocks' theorem (cf. Theorem 2.3.1 in \cite{O-S-S} chapter 1), $M\cong \mathcal{O}^{\oplus 2}_{\mathbb{P}^3}$. Thus by Lemma \ref{trivial-lemma}, $p_*(HN^{2}/HN^1)\cong \mathcal{O}^{\oplus 2}_{\mathbb{P}^4}$ and $HN^{2}/HN^1\cong p^*\mathcal{O}^{\oplus 2}_{\mathbb{P}^4}(1)$.
Since we have the exact sequence
\[0\longrightarrow HN^1\longrightarrow HN^2\longrightarrow HN^{2}/HN^1\longrightarrow 0\] and $HN^1=\mathcal{O}_{\mathbb{P}(T_{\mathbb{P}^4}(-1))}(-2)$, so we have the exact sequence
\[0\longrightarrow \mathcal{O}_{\mathbb{P}(T_{\mathbb{P}^4}(-1))}(-2)\longrightarrow HN^2\longrightarrow p^*\mathcal{O}^{\oplus 2}_{\mathbb{P}^4}(1)\longrightarrow 0.\]
 Since $$H^1(\mathbb{P}(T_{\mathbb{P}^4}(-1)), \mathcal{O}_{\mathbb{P}(T_{\mathbb{P}^4}(-1))}(-2)\otimes p^*\mathcal{O}^{\oplus 2}_{\mathbb{P}^4}(-1))=H^1(\mathbb{P}^4, p_*\mathcal{O}_{\mathbb{P}(T_{\mathbb{P}^4}(-1))}(-2)\otimes\mathcal{O}^{\oplus 2}_{\mathbb{P}^4}(-1))=0,$$ we have \[HN^2\cong \mathcal{O}_{\mathbb{P}(T_{\mathbb{P}^4}(-1))}(-2) \oplus p^*\mathcal{O}^{\oplus 2}_{\mathbb{P}^4}(1).\]
 	
By the same method used in Lemma \ref{deco-lem}, applying $p_{*}$ to $HN$-sequence, we have the following commutative diagram.
 	\begin{center}
 		$\xymatrix{
 			0\ar[r]& 	p^{*}\mathcal{O}^{\oplus 2}_{\mathbb{P}^4}(1) \ar[r] \ar[d]& p^{*}E\ar[r]\ar[d]^{id}& p^{*}W\ar[r]\ar[d]&0\\
 			0\ar[r] & \mathcal{O}_{\mathbb{P}(T_{\mathbb{P}^4}(-1))}(-2)\oplus p^{*}\mathcal{O}^{\oplus 2}_{\mathbb{P}^4}(1) \ar[r] &p^{*}E\ar[r] & p^{*}E/HN^2\ar[r]&0,\\
 		}$
 	\end{center}
 	where $W=p_{*}(p^{*}E/HN^2)$ .
 	
 	By the snake lemma, the following sequence is exact:
 	\[0\longrightarrow \mathcal{O}_{\mathbb{P}(T_{\mathbb{P}^4}(-1))}(-2)\longrightarrow p^{*}W\longrightarrow p^{*}E/HN^2\longrightarrow 0. \]
 	
 	According to the splitting type of $E$, the restriction of $p^{*}E/HN^2$ to the $q$-fiber at any point $l\in \mathbb{G}(1,4)$, $p^{*}E/HN^2|_{q^{-1}(l)}$ is trivial. Moreover, $\mathcal{O}_{\mathbb{P}(T_{\mathbb{P}^4}(-1))}(-2)|_{q^{-1}(l)}\cong \mathcal{O}_{q^{-1}(l)}(2)$. Since $\mathcal{H}om(\mathcal{O}_{q^{-1}(l)}(2), \mathcal{O}_{q^{-1}(l)})=0$, we have $$Hom(T_{\mathbb{P}(T_{\mathbb{P}^4}(-1))/\mathbb{P}^4},\mathcal{H}om(\mathcal{O}_{\mathbb{P}(T_{\mathbb{P}^4}(-1))}(-2), p^{*}E/HN^2))=0.$$ So by Descente-Lemma (cf. Lemma 2.1.2 in \cite{O-S-S} chapter 2), there is a vector bundle $E'$ over $\mathbb{P}^4$ such that $p^{*}E'\cong \mathcal{O}_{\mathbb{P}(T_{\mathbb{P}^4}(-1))}(-2)$, which is a contradiction.
 \end{proof}

  So we prove the following proposition:
  \begin{proposition}
 	The uniform vector bundle $E$ over $\mathbb{P}^4$ with splitting form of $(3;1,2,3;u_1,u_1-1,u_1-2)$ is isomorphic to either $$ \mathcal{O}_{\mathbb{P}^{4}}(u_1)\oplus\mathcal{O}_{\mathbb{P}^{4}}^{\oplus 2}(u_1-1)\oplus \mathcal{O}_{\mathbb{P}^{4}}^{\oplus 3}(u_1-2)$$ or $$ \mathcal{O}_{\mathbb{P}^4}(u_1) \oplus \mathcal{O}_{\mathbb{P}^4}(u_1-1)\oplus  T_{\mathbb{P}^4}(u_1-3).$$
  \end{proposition}
		
\vspace{1cm}

 \subsection{$k=3,r_1=2,r_2=1,r_3=3$}
 We suppose that
\begin{equation}\label{r1=2r2=1}
c_{p^{*}E}(T)=\overset{3}{\underset{i=1}{\prod}}S_i(T+u_{i}U,U,V).
\end{equation}
Without loss of generality, we may assume $u_3=0$. We also assume that
\[S_{2}(T,U,V)=T+a(U+V),\]
\[S_{1}(T,U,V)=T^2+b_1(U+V)T+b_2^0U^2+b_2^1UV+b_2^0V^2\] and
\[S_{3}(T,U,V)=T^3+n_1(U+V)T^2+(n_2^0U^2+n_2^1UV+n_2^0V^2)T+(n_3^0U^3+n_3^1U^2V+n_3^1UV^2+n_3^0V^3),\]
where $a, b_1, b_2^0, b_2^1, b_2^2, n_1, n_2^0, n_2^1, n_3^0, n_3^1$ are all constants.

Using computer calculation, we get only following three solutions of equation (\ref{r1=2r2=1}).
  \begin{center}
 	$(a,b_1,b_2^0,b_2^1,n_1,n_2^0,n_2^1,n_3^0,n_3^1)=(0,0,0,0,0,0,0,0,0)$,
 	
 	$(a,b_1,b_2^0,b_2^1,n_1,n_2^0,n_2^1,n_3^0,n_3^1)=(-1,0,0,0,1,1,1, 1, 1)$,
 	
 	$(a,b_1,b_2^0,b_2^1,n_1,n_2^0,n_2^1,n_3^0,n_3^1)=(0, -2, 0, 0,2, 4, 4, 8, 8)$.
 \end{center}
If $(a,b_1,b_2^0,b_2^1,n_1,n_2^0,n_2^1,n_3^0,n_3^1)=(0,0,0,0,0,0,0,0,0)$, then by the similar arguments in the proof of Proposition \ref{(0,0,0)case}, Proposition \ref{0-1}, Proposition \ref{-2 case},
%
%
%
%
we have the theorem as follows.
 \begin{proposition}
 	The uniform vector bundle $E$ over $\mathbb{P}^4$ with splitting form of $(3;2,1,3;u_1,u_1-1,u_1-2)$ is isomorphic to either $$ \mathcal{O}^{\oplus 2}_{\mathbb{P}^{4}}(u_1)\oplus\mathcal{O}_{\mathbb{P}^{4}}(u_1-1)\oplus \mathcal{O}_{\mathbb{P}^{4}}^{\oplus 3}(u_1-2)$$ or $$ \mathcal{O}^{\oplus 2}_{\mathbb{P}^4}(u_1) \oplus  T_{\mathbb{P}^4}(u_1-3).$$
  \end{proposition}

 \vspace{1cm}

 \subsection{$k=3,r_1=1,r_2=3,r_3=2$}
 We suppose that
\begin{equation}\label{r1=1r2=3}
c_{p^{*}E}(T)=\overset{3}{\underset{i=1}{\prod}}S_i(T+u_{i}U,U,V).
\end{equation}
Without loss of generality, we may assume $u_3=0$. We also assume that
\[S_{1}(T,U,V)=T+a(U+V),\]
\[S_{3}(T,U,V)=T^2+b_1(U+V)T+b_2^0U^2+b_2^1UV+b_2^0V^2\] and
\[S_{2}(T,U,V)=T^3+n_1(U+V)T^2+(n_2^0U^2+n_2^1UV+n_2^0V^2)T+(n_3^0U^3+n_3^1U^2V+n_3^1UV^2+n_3^0V^3),\]
where $a, b_1, b_2^0, b_2^1, b_2^2, n_1, n_2^0, n_2^1, n_3^0, n_3^1$ are all constants.

Using computer calculation, we get only following three solutions of equation (\ref{r1=1r2=3}).
  \begin{center}
 	$(a,b_1,b_2^0,b_2^1,n_1,n_2^0,n_2^1,n_3^0,n_3^1)=(0,0,0,0,0,0,0,0,0)$,
 	
 	$(a,b_1,b_2^0,b_2^1,n_1,n_2^0,n_2^1,n_3^0,n_3^1)=(-1,0,0,0,1,1,1,1,1)$,
 	
 	$(a,b_1,b_2^0,b_2^1,n_1,n_2^0,n_2^1,n_3^0,n_3^1)=(0,1,0,0,-1,1,1, -1, -1)$.
 \end{center}
If $(a,b_1,b_2^0,b_2^1,n_1,n_2^0,n_2^1,n_3^0,n_3^1)=(0,0,0,0,0,0,0,0,0)$, then by the similar arguments in the proof of Proposition \ref{(0,0,0)case}, $E\cong \mathcal{O}_{\mathbb{P}^{4}}(2)\oplus\mathcal{O}_{\mathbb{P}^{4}}^{\oplus 3}(1)\oplus \mathcal{O}_{\mathbb{P}^{4}}^{\oplus 2}.$

If $(a,b_1,b_2^0,b_2^1,n_1,n_2^0,n_2^1,n_3^0,n_3^1)=(0,1,0,0,-1,1,1, -1, -1)$, then by same arguments in Proposition \ref{0-1}, $E$ has $\mathcal{O}_{\mathbb{P}^4}(2)$ as subbundle. Finally, it is easy to get that $E\cong \mathcal{O}_{\mathbb{P}^4}(2)\oplus\Omega_{\mathbb{P}^4}(2)\oplus \mathcal{O}_{\mathbb{P}^4}$.

If $(a,b_1,b_2^0,b_2^1,n_1,n_2^0,n_2^1,n_3^0,n_3^1)=(-1,0,0,0,1,1,1,1,1)$, then by the same arguments in Proposition \ref{(-1,0,0)case}, we can get that $E\cong T_{\mathbb{P}^4}\oplus \mathcal{O}^{\oplus2}_{\mathbb{P}^4}$.

So the following proposition holds.
\begin{proposition}
   	The uniform vector bundle $E$ over $\mathbb{P}^4$ with splitting form of $(3;1,3,2;u_1,u_1-1,u_1-2)$ is isomorphic to $$\mathcal{O}_{\mathbb{P}^{4}}(u_1)\oplus\mathcal{O}_{\mathbb{P}^{4}}^{\oplus 3}(u_1-1)\oplus \mathcal{O}_{\mathbb{P}^{4}}^{\oplus 2}(u_1-2),$$ $$\mathcal{O}_{\mathbb{P}^4}(u_1)\oplus\Omega_{\mathbb{P}^4}(u_1)\oplus \mathcal{O}_{\mathbb{P}^4}(u_1-2)$$ or $$T_{\mathbb{P}^4}(u_1-2)\oplus \mathcal{O}^{\oplus2}_{\mathbb{P}^4}(u_1-2).$$
\end{proposition}

\vspace{1cm}

\subsection{$k=3,r_1=r_2=1,r_3=4$}
 We suppose that
\begin{equation}\label{r1=1r2=1}
c_{p^{*}E}(T)=\overset{3}{\underset{i=1}{\prod}}S_i(T+u_{i}U,U,V).
\end{equation}
Without loss of generality, we may assume $u_3=0$. We also assume that
\[S_{1}(T,U,V)=T+a(U+V),\]
\[S_{2}(T,U,V)=T+b(U+V)\] and
\begin{multline*}
S_{3}(T,U,V)=T^4+n_1(U+V)T^3+[n_2^0(U^2+V^2)+n_2^1UV]T^2+[n_3^0(U^3+V^3)+n_3^1(U^2V+UV^2)]T\\+[n_4^0(U^4+V^4)+n_4^1(U^3V+UV^3)+n_4^2U^2V^2],
\end{multline*}
where $a, b, n_1, n_2^0, n_2^1, n_3^0, n_3^1,n_4^0, n_4^1, n_4^2$ are all constants.

Using computer calculation, we get only following three solutions of equation (\ref{r1=1r2=1}).
  \begin{center}
 	$(a,b,n_1,n_2^0,n_2^1,n_3^0,n_3^1)=(0,0,0,0,0,0,0)$,
 	
 	$(a,b,n_1,n_2^0,n_2^1,n_3^0,n_3^1)=(-2,0,2,4,4,8,8)$,
 	
 	$(a,b,n_1,n_2^0,n_2^1,n_3^0,n_3^1)=(0,-1,1,1,1,1,1)$.
 \end{center}

If $(a,b,n_1,n_2^0,n_2^1,n_3^0,n_3^1)=(0,0,0,0,0,0,0)$, then by the similar arguments in the proof of Proposition \ref{(0,0,0)case}, $E\cong \mathcal{O}_{\mathbb{P}^{4}}(2)\oplus\mathcal{O}_{\mathbb{P}^{4}}(1)\oplus \mathcal{O}^{\oplus 4}_{\mathbb{P}^{4}}.$

If $(a,b,n_1,n_2^0,n_2^1,n_3^0,n_3^1)=(0,-1,1,1,1,1,1)$, then by same arguments in
Proposition \ref{0-1}, it is easy to get that $E\cong \mathcal{O}_{\mathbb{P}^4}(2)\oplus T_{\mathbb{P}^4}(-1)\oplus \mathcal{O}_{\mathbb{P}^4}$.

If $(a,b,n_1,n_2^0,n_2^1,n_3^0,n_3^1)=(-2,0,2,4,4,8,8)$, by the same argument in Proposition \ref{-2 case}, there does not exist such a uniform bundle.

So the following proposition holds.
 \begin{proposition}
 The uniform vector bundle $E$ over $\mathbb{P}^4$ with splitting form of $(3;1,1,4;u_1,u_1-1,u_1-2)$ is isomorphic to either $$ \mathcal{O}_{\mathbb{P}^{4}}(u_1)\oplus\mathcal{O}_{\mathbb{P}^{4}}(u_1-1)\oplus \mathcal{O}^{\oplus 4}_{\mathbb{P}^{4}}(u_1-2)$$ or $$ \mathcal{O}_{\mathbb{P}^4}(u_1)\oplus T_{\mathbb{P}^4}(u_1-3)\oplus \mathcal{O}_{\mathbb{P}^4}(u_1-2).$$
 \end{proposition}

 \vspace{1cm}

 \subsection{$k=3,r_1=r_3=1, r_2=4$}
 We suppose that
\begin{equation}\label{r1=1r3=1}
c_{p^{*}E}(T)=\overset{3}{\underset{i=1}{\prod}}S_i(T+u_{i}U,U,V).
\end{equation}
Without loss of generality, we may assume $u_3=0$. We also assume that
\[S_{1}(T,U,V)=T+a(U+V),\]
\[S_{3}(T,U,V)=T+b(U+V)\] and
\begin{multline*}
S_{2}(T,U,V)=T^4+n_1(U+V)T^3+[n_2^0(U^2+V^2)+n_2^1UV]T^2+[n_3^0(U^3+V^3)+n_3^1(U^2V+UV^2)]T\\+[n_4^0(U^4+V^4)+n_4^1(U^3V+UV^3)+n_4^2U^2V^2],
\end{multline*}
where $a, b, n_1, n_2^0, n_2^1, n_3^0, n_3^1,n_4^0, n_4^1, n_4^2$ are all constants.

Using computer calculation, we get only following three solutions of equation (\ref{r1=1r2=1}).
  \begin{center}
 	$(a,b,n_1,n_2^0,n_2^1,n_3^0,n_3^1)=(0,0,0,0,0,0,0)$,
 	
 	$(a,b,n_1,n_2^0,n_2^1,n_3^0,n_3^1)=(-1,0,1,1,1,1,1)$,
 	
 	$(a,b,n_1,n_2^0,n_2^1,n_3^0,n_3^1)=(0,1,-1,1,1,-1,-1)$.
 \end{center}

By similar arguments, we can get the following proposition.

\begin{proposition}
  The uniform vector bundle $E$ over $\mathbb{P}^4$ with splitting form of $(3;1,4,1;u_1,u_1-1,u_1-2)$ is isomorphic to either $$ \mathcal{O}_{\mathbb{P}^{4}}(u_1)\oplus\mathcal{O}^{\oplus 4}_{\mathbb{P}^{4}}(u_1-1)\oplus \mathcal{O}_{\mathbb{P}^{4}}(u_1-2),$$
   $$ T_{\mathbb{P}^4}(u_1-3)\oplus \mathcal{O}_{\mathbb{P}^4}(u_1-2)\oplus \mathcal{O}_{\mathbb{P}^4}(u_1-3)$$
  or $$ \Omega^1_{\mathbb{P}^4}(u_1-1)\oplus \mathcal{O}_{\mathbb{P}^4}(u_1-2)\oplus \mathcal{O}_{\mathbb{P}^4}(u_1-1) .$$
 \end{proposition}
%
%
%


 \vspace{1cm}

 \subsection{$k=4, r_1=r_2=r_3=1, r_4=3$}
  We suppose that
\begin{equation}\label{r4=3}
c_{p^{*}E}(T)=\overset{4}{\underset{i=1}{\prod}}S_i(T+u_{i}U,U,V).
\end{equation}

Without loss of generality, we may assume $u_4=0$.
We assume that
\[S_{1}(T,U,V)=T+a(U+V),\]
\[S_{2}(T,U,V)=T+b(U+V),\]
\[S_{3}(T,U,V)=T+c(U+V)\] and
\[
S_{4}(T,U,V)=T^3+n_1(U+V)T^2+[n_2^0(U^2+V^2)+n_2^1UV]T+[n_3^0(U^3+V^3)+n_3^1(U^2V+UV^2)],
\]
where $a, b, c, n_1, n_2^0, n_2^1, n_3^0, n_3^1$ are all constants.

Using computer calculation, we get only following three solutions of equation (\ref{r4=3}).
  \begin{center}
 	 $(a,b,c,n_1,n_2^0,n_2^1,n_3^0,n_3^1)=(0,0,0,0,0,0,0,0),$
 	
 	$(a,b,c,n_1,n_2^0,n_2^1,n_3^0,n_3^1)=(0,0,-1,1,1,1,1,1),$
 	
 	$(a,b,c,n_1,n_2^0,n_2^1,n_3^0,n_3^1)=(0,-2,0,2,4,4,8,8),$
 	
 	$(a,b,c,n_1,n_2^0,n_2^1,n_3^0,n_3^1)=(-3,0,0,3,9,9,27,27).$
 \end{center}

If $(a,b,c,n_1,n_2^0,n_2^1,n_3^0,n_3^1)=(0,0,0,0,0,0,0,0)$, then by the similar arguments in the proof of Proposition \ref{(0,0,0)case}, $E\cong \mathcal{O}_{\mathbb{P}^{4}}(3)\oplus \mathcal{O}_{\mathbb{P}^{4}}(2)\oplus\mathcal{O}_{\mathbb{P}^{4}}(1)\oplus \mathcal{O}^{\oplus 3}_{\mathbb{P}^{4}}.$

If $(a,b,c,n_1,n_2^0,n_2^1,n_3^0,n_3^1)=(0,0,-1,1,1,1,1,1),$ then by same arguments in
Proposition \ref{0-1}, it is easy to get that  $E\cong T_{\mathbb{P}^4}(-1)\oplus \mathcal{O}_{\mathbb{P}^4}(3)\oplus\mathcal{O}_{\mathbb{P}^4}(2)$.

If $(a,b,c,n_1,n_2^0,n_2^1,n_3^0,n_3^1)=(0,-2,0,2,4,4,8,8)$ or $(-3,0,0,3,9,9,27,27)$, by the same argument in Proposition \ref{-2 case}, there does not exist such a uniform bundle.

So the following proposition holds.
 \begin{proposition}
 The uniform vector bundle $E$ over $\mathbb{P}^4$ with splitting form of $(4;1,1,1,3;u_1,u_1-1,u_1-2,u_1-3)$ is isomorphic to either $$ \mathcal{O}_{\mathbb{P}^{4}}(u_1)\oplus \mathcal{O}_{\mathbb{P}^{4}}(u_1-1)\oplus\mathcal{O}_{\mathbb{P}^{4}}(u_1-2)\oplus \mathcal{O}^{\oplus 3}_{\mathbb{P}^{4}}(u_1-3)$$ or $$ E\cong T_{\mathbb{P}^4}(u_1-3)\oplus \mathcal{O}_{\mathbb{P}^4}(u_1)\oplus\mathcal{O}_{\mathbb{P}^4}(u_1-1).$$
 \end{proposition}


%
 \vspace{1cm}

\subsection{$k=4,r_1=r_2=r_4=1, r_3=3$}
We suppose that
\begin{equation}\label{r3=3}
c_{p^{*}E}(T)=\overset{4}{\underset{i=1}{\prod}}S_i(T+u_{i}U,U,V).
\end{equation}

Without loss of generality, we may assume $u_4=0$.
We assume that
\[S_{1}(T,U,V)=T+a(U+V),\]
\[S_{2}(T,U,V)=T+b(U+V),\]
\[S_{4}(T,U,V)=T+c(U+V)\] and
\[
S_{3}(T,U,V)=T^3+n_1(U+V)T^2+[n_2^0(U^2+V^2)+n_2^1UV]T+[n_3^0(U^3+V^3)+n_3^1(U^2V+UV^2)],
\]
where $a, b, c, n_1, n_2^0, n_2^1, n_3^0, n_3^1$ are all constants.

Using computer calculation, we get only following three solutions of equation (\ref{r3=3}).
  \begin{center}
 	 $(a,b,c,n_1,n_2^0,n_2^1,n_3^0,n_3^1)=(0,0,0,0,0,0,0,0),$
 	
 	$(a,b,c,n_1,n_2^0,n_2^1,n_3^0,n_3^1)=(-2,0,0,2,4,4,8,8),$
 	
 	$(a,b,c,n_1,n_2^0,n_2^1,n_3^0,n_3^1)=(0,-1,0,1,1,1,1,1),$
 	
 	$(a,b,c,n_1,n_2^0,n_2^1,n_3^0,n_3^1)=(0,0,1,-1,1,1,-1,-1).$
 \end{center}

If $(a,b,c,n_1,n_2^0,n_2^1,n_3^0,n_3^1)=(0,0,0,0,0,0,0,0)$, then by the similar arguments in the proof of Proposition \ref{(0,0,0)case}, $E\cong \mathcal{O}_{\mathbb{P}^{4}}(3)\oplus \mathcal{O}_{\mathbb{P}^{4}}(2)\oplus\mathcal{O}^{\oplus 3}_{\mathbb{P}^{4}}(1)\oplus \mathcal{O}_{\mathbb{P}^{4}}.$

If $(a,b,c,n_1,n_2^0,n_2^1,n_3^0,n_3^1)=(0,0,-1,1,1,1,1,1),$ then by same arguments in
Proposition \ref{0-1}, it is easy to get that  $E\cong T_{\mathbb{P}^4}\oplus \mathcal{O}_{\mathbb{P}^4}(3)\oplus\mathcal{O}_{\mathbb{P}^4}$.

If $(a,b,c,n_1,n_2^0,n_2^1,n_3^0,n_3^1)=(0,0,1,-1,1,1,-1,-1),$ then by similar arguments in
Proposition \ref{0-1}, it is easy to get that  $E\cong \Omega^1_{\mathbb{P}^4}(2)\oplus \mathcal{O}_{\mathbb{P}^4}(3)\oplus\mathcal{O}_{\mathbb{P}^4}(2)$.

If $(a,b,c,n_1,n_2^0,n_2^1,n_3^0,n_3^1)=(-2,0,0,2,4,4,8,8)$, by the same argument in Proposition \ref{-2 case}, there does not exist such a uniform bundle.

So the following proposition holds.
 \begin{proposition}
 The uniform vector bundle $E$ over $\mathbb{P}^4$ with splitting form of $(4;1,1,3,1;u_1,u_1-1,u_1-2,u_1-3)$ is isomorphic to $$ E\cong \mathcal{O}_{\mathbb{P}^{4}}(u_1)\oplus \mathcal{O}_{\mathbb{P}^{4}}(u_1-1)\oplus\mathcal{O}^{\oplus 3}_{\mathbb{P}^{4}}(u_1-2)\oplus \mathcal{O}_{\mathbb{P}^{4}}(u_1-3),$$
 $$E\cong T_{\mathbb{P}^4}(u_1-3)\oplus \mathcal{O}_{\mathbb{P}^4}(u_1)\oplus\mathcal{O}_{\mathbb{P}^4}(u_1-3)$$
 or $$E\cong \Omega^1_{\mathbb{P}^4}(u_1-1)\oplus \mathcal{O}_{\mathbb{P}^4}(u_1)\oplus\mathcal{O}_{\mathbb{P}^4}(u_1-1).$$
 \end{proposition}

%
%
%
%
%
%
%
%
%
%
%
%
%
%
 \vspace{1cm}

 \subsection{$r_i\leq 2$ for all $i$ and the conclusion for rank $6$ uniform bundles}
In this case,
\begin{equation}\label{ri<3}
c_{p^{*}E}(T)=\overset{k}{\underset{i=1}{\prod}}S_i(T+u_{i}U,U,V),
\end{equation}
where $k\le 6$.
Since the degree of $S_i(T,U,V)$ is less than or equal to $2$, we calculate all the cases one by one to get that there is only the zero solution of equation (\ref{ri<3}) if we set the similar assumption as before. Thus by the similar arguments in the proof of Proposition \ref{(0,0,0)case}, the bundle $E$ is the direct sum of line bundles.

So we can get the following theorem.
 \begin{theorem}\label{r6k5}
 	Any uniform $6$-bundle with $k\geq 5$ over $\mathbb{P}^4$ is a direct sum of line bundles.
 \end{theorem}

For other cases, we only need to take dual and tensor a suitable line bundle to boil down to the cases in the former sections.

In conclusion, we have proved the following theorem.

\begin{theorem}\label{r6}
	Any uniform $6$-bundle over $\mathbb{P}^4$ is homogeneous.
\end{theorem}
 \vspace{1cm}

 \section{Uniform $7$-bundles over $\mathbb{P}^4$}

 \subsection{$k=2, r_1=2, r_2=5$}
 In this case, we may assume
 \begin{center}
 	$S_{1}(T,U,V)=T^2+t_1(U+V)T+[t_2^0(U^2+V^2)+t_2^1UV]$,
  \end{center}
where $t_1,t_2^0,t_2^1$ are all constants.

There are only two solutions for the Chern polynomial equation if we set $u_2=0$.
\[(t_1,t_2^0,t_2^1)=(0,0,0)~ \text{and}~ (-1,0,0).\]

\begin{proposition}
	The uniform vector bundle $E$ over $\mathbb{P}^4$ with splitting type $(2;2,5;u_1,u_1-1)$ is either $$\mathcal{O}^{\oplus 2}_{\mathbb{P}^{4}}(u_1)\oplus \mathcal{O}^{\oplus 5}_{\mathbb{P}^{4}}(u_1-1)$$ or $$T_{\mathbb{P}^4}(u_1-2)\oplus \mathcal{O}_{\mathbb{P}^4}(u_1)\oplus\mathcal{O}^{\oplus 2}_{\mathbb{P}^4}(u_1-1)$$.
\end{proposition}
\begin{proof}
	If $(t_1,t_2^0,t_2^1)=(0,0,0)$, then by the similar arguments in the proof of Proposition \ref{(0,0,0)case}, $E\cong \mathcal{O}^{\oplus 2}_{\mathbb{P}^{4}}(u_1)\oplus \mathcal{O}^{\oplus 5}_{\mathbb{P}^{4}}(u_1-1).$
	
	If $(t_1,t_2^0,t_2^1)=(-1,0,0)$, then by the similar arguments in the proof of Proposition \ref{(-1,0,0)case}, $E\cong T_{\mathbb{P}^4}(u_1-2)\oplus \mathcal{O}_{\mathbb{P}^4}(u_1)\oplus\mathcal{O}^{\oplus 2}_{\mathbb{P}^4}(u_1-1)$.
\end{proof}

 \vspace{1cm}

 \subsection{$k=2, r_1=3, r_2=4$}
Suppose that
 \begin{center}
 	$S_{1}(T,U,V)=T^3+t_1(U+V)T^2+[t_2^0(U^2+V^2)+t_2^1UV]T+[t_3^0(U^3+V^3)+t_3^1(U^2V+UV^2)]$.
 \end{center}

There are four solutions for the Chern polynomial equation if we set $u_2=0$ (See the Appendix $(3,4)$-case).
 \begin{center}
 	$(t_1,t_2^0,t_2^1,t_3^0,t_3^1)=(-2,2,3,0,-1),$
 	
 	$(t_1,t_2^0,t_2^1,t_3^0,t_3^1)=(-1,0,0,0,0),$
 	
 	$(t_1,t_2^0,t_2^1,t_3^0,t_3^1)=(-1,1,1,-1,-1),$
 	
 	$(t_1,t_2^0,t_2^1,t_3^0,t_3^1)=(0,0,0,0,0)$.
 \end{center}
 \begin{proposition}
 	The uniform vector bundle $E$ over $\mathbb{P}^4$ with splitting type $(2;3,4;u_1,u_1-1)$ is $$\mathcal{O}^{\oplus 3}_{\mathbb{P}^{4}}(u_1)\oplus \mathcal{O}^{\oplus 4}_{\mathbb{P}^{4}}(u_1-1),$$
  $$T_{\mathbb{P}^4}(u_1-2)\oplus \mathcal{O}^{\oplus 2}_{\mathbb{P}^4}(u_1)\oplus\mathcal{O}_{\mathbb{P}^4}(u_1-1),$$
  $$(\bigwedge ^{2}T_{\mathbb{P}^4}(-1))\otimes \mathcal{O}_{\mathbb{P}^4}(u_1-1)\oplus \mathcal{O}_{\mathbb{P}^4}(u_1-1)$$
  or $$ \Omega^1_{\mathbb{P}^4}(u_1+1)\oplus \mathcal{O}^{\oplus 3}_{\mathbb{P}^4}(u_1-1).$$
 \end{proposition}
\begin{proof}
	If $(t_1,t_2^0,t_2^1,t_3^0,t_3^1)=(0,0,0,0,0)$, then by the similar arguments in the proof of Proposition \ref{(0,0,0)case}, $E\cong \mathcal{O}^{\oplus 3}_{\mathbb{P}^{4}}(u_1)\oplus \mathcal{O}^{\oplus 4}_{\mathbb{P}^{4}}(u_1-1).$
	
	If $(t_1,t_2^0,t_2^1,t_3^0,t_3^1)=(-1,0,0,0,0)$, then by the similar arguments in the proof of Proposition \ref{(-1,0,0)case}, $E\cong T_{\mathbb{P}^4}(u_1-2)\oplus \mathcal{O}^{\oplus 2}_{\mathbb{P}^4}(u_1)\oplus\mathcal{O}_{\mathbb{P}^4}(u_1-1)$.
	
	If $(t_1,t_2^0,t_2^1,t_3^0,t_3^1)=(-2,2,3,0,-1)$ ,then by the similar arguments in the proof of Proposition \ref{exterior T}, $HN^1|_{p^{-1}(x)}\cong T_{p^{-1}(x)}(-1)$ and $p^{*}E/HN^1|_{p^{-1}(x)}\cong \mathcal{O}_{p^{-1}(x)}\oplus \Omega^1_{\mathbb{P}^3}(2)$. Finally, we can get that $E\cong (\bigwedge ^{2}T_{\mathbb{P}^4}(-1))\otimes \mathcal{O}_{\mathbb{P}^4}(u_1-1)\oplus \mathcal{O}_{\mathbb{P}^4}(u_1-1).$
	
	If $(t_1,t_2^0,t_2^1,t_3^0,t_3^1)=(-1,1,1,-1,-1)$, then by the similar arguments in the proof of Proposition \ref{(-1,0,0)case} for $E^{*}$, we can get $E\cong \Omega^1_{\mathbb{P}^4}(u_1+1)\oplus \mathcal{O}^{\oplus 3}_{\mathbb{P}^4}(u_1-1)$.
\end{proof}


  \subsection{$k=3, r_1=r_2=1, r_3=5$}
 We may assume
\[S_{1}(T,U,V)=T+a(U+V)\] and
\[S_{2}(T,U,V)=T+b(U+V),\]
where $a, b$ are constants.

There are two solutions for the Chern polynomial equation if we set $u_3=0$.
 \begin{center}
 	$(a,b)=(-2,0), (0,0), (0,-1)$.
 \end{center}

\begin{proposition}
	The uniform vector bundle $E$ over $\mathbb{P}^4$ with splitting type $(3;1,1,5;u_1,u_1-1,u_1-2)$ is either $$E\cong \mathcal{O}_{\mathbb{P}^{4}}(u_1)\oplus \mathcal{O}_{\mathbb{P}^{4}}(u_1-1)\oplus \mathcal{O}^{\oplus 5}_{\mathbb{P}^{4}}(u_1-2)$$ or $$T_{\mathbb{P}^4}(u_1-3)\oplus \mathcal{O}_{\mathbb{P}^4}(u_1)\oplus\mathcal{O}^{\oplus 2}_{\mathbb{P}^4}(u_1-2).$$
\end{proposition}
\begin{proof}
	If $(a,b)=(0,0)$, then by the similar arguments in the proof of Proposition \ref{(0,0,0)case}, $E\cong \mathcal{O}_{\mathbb{P}^{4}}(u_1)\oplus \mathcal{O}_{\mathbb{P}^{4}}(u_1-1)\oplus \mathcal{O}^{\oplus 5}_{\mathbb{P}^{4}}(u_1-2).$

If $(a,b)=(0,-1)$, then by the similar arguments in the proof of Proposition \ref{(-1,0,0)case}, we have $E\cong T_{\mathbb{P}^4}(u_1-3)\oplus \mathcal{O}_{\mathbb{P}^4}(u_1)\oplus\mathcal{O}^{\oplus 2}_{\mathbb{P}^4}(u_1-2)$.
	
	If $a=-2$, by the same argument in Proposition \ref{-2 case}, there does not exist such a uniform bundle.
\end{proof}

 \vspace{1cm}

 \subsection{$k=3,r_1=r_3=1, r_2=5$}
 We may assume
\[S_{1}(T,U,V)=T+a(U+V)\] and
\[S_{2}(T,U,V)=T+b(U+V),\] where $a, b$ are constants.

There are three solutions for the Chern polynomial equation if we set $u_3=0$.
\begin{center}
	$(a,b)=(-1,0), (0,0), (0,1).$
\end{center}

Using the similar arguments, we have following result.
\begin{proposition}
		The uniform vector bundle $E$ over $\mathbb{P}^4$ with splitting type $(3;1,5,1;u_1,u_1-1,u_1-2)$ is $$\mathcal{O}_{\mathbb{P}^{4}}(u_1)\oplus \mathcal{O}^{\oplus 5}_{\mathbb{P}^{4}}(u_1-1)\oplus \mathcal{O}_{\mathbb{P}^{4}}(u_1-2),$$ $$T_{\mathbb{P}^4}(u_1-2)\oplus \mathcal{O}^{\oplus 2}_{\mathbb{P}^4}(u_1-1)\oplus\mathcal{O}_{\mathbb{P}^4}(u_1-2)$$ or $$\Omega^1_{\mathbb{P}^4}(u_1)\oplus \mathcal{O}_{\mathbb{P}^4}(u_1)\oplus\mathcal{O}^{\oplus 2}_{\mathbb{P}^4}(u_1-1).$$
\end{proposition}
%
%

\vspace{1cm}

\subsection{$k=3, r_1=1,r_2=2,r_3=4$}
 We may assume
\[S_{1}(T,U,V)=T+a(U+V)\] and
\[S_{2}(T,U,V)=T^2+b_1(U+V)T+b_2^0(U^2+V^2)+b_2^1UV,\] where $a, b_1, b_2^0, b_2^1$ are constants.

There are two solutions for the Chern polynomial equation if we set $u_3=0$.
\begin{center}
	$(a,b_1,b_2^0,b_2^1)=(-2,0,0,0), (0,0,0,0), (0,-1,0,0)$.
\end{center}

Using the similar arguments, we have following result.
\begin{proposition}
		The uniform vector bundle $E$ over $\mathbb{P}^4$ with splitting type $(3;1,2,4;u_1,u_1-1,u_1-2)$ is either $$\mathcal{O}_{\mathbb{P}^{4}}(u_1)\oplus \mathcal{O}^{\oplus 2}_{\mathbb{P}^{4}}(u_1-1)\oplus \mathcal{O}^{\oplus 4}_{\mathbb{P}^{4}}(u_1-2)$$ or $$T_{\mathbb{P}^4}(u_1-3)\oplus \mathcal{O}_{\mathbb{P}^4}(u_1)\oplus\mathcal{O}_{\mathbb{P}^4}(u_1-1))\oplus\mathcal{O}_{\mathbb{P}^4}(u_1-2).$$
\end{proposition}

%

 \vspace{1cm}

 \subsection{$k=3, r_1=2,r_2=1,r_3=4$}
  We may assume
\[S_{2}(T,U,V)=T+a(U+V)\] and
\[S_{1}(T,U,V)=T^2+b_1(U+V)T+b_2^0(U^2+V^2)+b_2^1UV,\] where $a, b_1, b_2^0, b_2^1$ are constants.

There are two solutions for the Chern polynomial equation if we set $u_3=0$.
  \begin{center}
 	$(a,b_1,b_2^0,b_2^1)=(0,0,0,0), (-1,0,0,0)$.
 \end{center}

Using the similar arguments, we have following result.
\begin{proposition}
		The uniform vector bundle $E$ over $\mathbb{P}^4$ with splitting type $(3;2,1,4;u_1,u_1-1,u_1-2)$ is either $$\mathcal{O}^{\oplus 2}_{\mathbb{P}^{4}}(u_1)\oplus \mathcal{O}_{\mathbb{P}^{4}}(u_1-1)\oplus \mathcal{O}^{\oplus 4}_{\mathbb{P}^{4}}(u_1-2)$$ or $$T_{\mathbb{P}^4}(u_1-3)\oplus \mathcal{O}^{\oplus 2}_{\mathbb{P}^4}(u_1)\oplus\mathcal{O}_{\mathbb{P}^4}(u_1-2).$$
\end{proposition}

\vspace{1cm}

 \subsection{$k=3, r_1=2, r_2=4, r_3=1$}
  We may assume
\[S_{3}(T,U,V)=T+a(U+V)\] and
\[S_{1}(T,U,V)=T^2+b_1(U+V)T+b_2^0(U^2+V^2)+b_2^1UV,\] where $a, b_1, b_2^0, b_2^1$ are constants.

There are three solutions for the Chern polynomial equation if we set $u_3=0$.
  \begin{center}
 	$(a,b_1,b_2^0,b_2^1)=(0,-1,0,0), (0,0,0,0), (1,0,0,0)$.
 \end{center}

Using the similar arguments, we have following result.
\begin{proposition}
		The uniform vector bundle $E$ over $\mathbb{P}^4$ with splitting type $(3;2,4,1;u_1,u_1-1,u_1-2)$ is $$\mathcal{O}^{\oplus 2}_{\mathbb{P}^{4}}(u_1)\oplus \mathcal{O}^{\oplus 4}_{\mathbb{P}^{4}}(u_1-1)\oplus \mathcal{O}_{\mathbb{P}^{4}}(u_1-2),$$ $$T_{\mathbb{P}^4}(u_1-2)\oplus \mathcal{O}_{\mathbb{P}^4}(u_1)\oplus\mathcal{O}_{\mathbb{P}^4}(u_1-1))\oplus\mathcal{O}_{\mathbb{P}^4}(u_1-2)$$ or $$\Omega^1_{\mathbb{P}^4}(u_1)\oplus \mathcal{O}^{\oplus 2}_{\mathbb{P}^4}(u_1)\oplus\mathcal{O}_{\mathbb{P}^4}(u_1-1).$$
\end{proposition}
%
%
%

 \vspace{1cm}

 \subsection{$k=3,r_1=r_2=2,r_3=3$}
  We may assume
\[S_1(T,U,V)=T^2+a_1(U+V)T+a_2^0(U^2+V^2)+a_2^1UV\] and
\[S_2(T,U,V)=T^2+e_1(U+V)T+e_2^0(U^2+V^2)+e_2^1UV,\] where $a_1, a_2^0,a_2^1,e_1,e_2^0,e_2^1$ are constants.

There are three solutions for the Chern polynomial equation if we set $u_3=0$.
   \begin{center}
 	$(a_1,a_2^0,a_2^1,e_1,e_2^0,e_2^1)=(0,0,0,0,0,0),(0,0,0,-1,0,0),(-2,0,0,0,0,0)$.
 \end{center}

Using the similar arguments, we have following result.
\begin{proposition}
		The uniform vector bundle $E$ over $\mathbb{P}^4$ with splitting type $(3;2,2,3;u_1,u_1-1,u_1-2)$ is $$\mathcal{O}^{\oplus 2}_{\mathbb{P}^{4}}(u_1)\oplus \mathcal{O}^{\oplus 2}_{\mathbb{P}^{4}}(u_1-1)\oplus \mathcal{O}^{\oplus 3}_{\mathbb{P}^{4}}(u_1-2)$$ or $$T_{\mathbb{P}^4}(u_1-3)\oplus \mathcal{O}^{\oplus 2}_{\mathbb{P}^4}(u_1)\oplus\mathcal{O}_{\mathbb{P}^4}(u_1-1).$$
\end{proposition}

 %
%
%
%
%
%

 \vspace{1cm}

 \subsection{$k=3, r_1=r_3=2, r_2=3$}
  We may assume
\[S_1(T,U,V)=T^2+a_1(U+V)T+a_2^0(U^2+V^2)+a_2^1UV\] and
\[S_2(T,U,V)=T^2+e_1(U+V)T+e_2^0(U^2+V^2)+e_2^1UV,\] where $a_1, a_2^0,a_2^1,e_1,e_2^0,e_2^1$ are constants.

There are three solutions for the Chern polynomial equation if we set $u_3=0$.
   \begin{center}
 	$(a_1,a_2^0,a_2^1,e_1,e_2^0,e_2^1)=(0,0,0,0,0,0),(0,0,0,1,0,0),(-1,0,0,0,0,0)$.
 \end{center}

Using the similar arguments, we have following result.
\begin{proposition}
		The uniform vector bundle $E$ over $\mathbb{P}^4$ with splitting type $(3;2,3,2;u_1,u_1-1,u_1-2)$ is $$\mathcal{O}^{\oplus 2}_{\mathbb{P}^{4}}(u_1)\oplus \mathcal{O}^{\oplus 3}_{\mathbb{P}^{4}}(u_1-1)\oplus \mathcal{O}^{\oplus 2}_{\mathbb{P}^{4}}(u_1-2),$$ $$T_{\mathbb{P}^4}(u_1-2)\oplus \mathcal{O}_{\mathbb{P}^4}(u_1)\oplus \mathcal{O}^{\oplus 2}_{\mathbb{P}^4}(u_1-2)$$ or $$\Omega^1_{\mathbb{P}^4}(u_1)\oplus \mathcal{O}^{\oplus 2}_{\mathbb{P}^4}(u_1)\oplus\mathcal{O}_{\mathbb{P}^4}(u_1-2).$$
\end{proposition}

%
%
%
%
%
%

 \vspace{1cm}

\subsection{$k=3, r_1=1,r_2=r_3=3$}
We may assume
\[S_1(T,U,V)=T+a(U+V),\]
\[S_2(T,U,V)=T^3+b_1(U+V)T^2+(b_2^0U^2+b_2^1UV+b_2^0V^2)T+b_3^0(U^3+V^3)+b_3^1(U^2V+UV^2)\]
and \[S_3(T,U,V)=T^3+c_1(U+V)T^2+(c_2^0U^2+c_2^1UV+c_2^0V^2)T+c_3^0(U^3+V^3)+c_3^1(U^2V+UV^2),\]
 where $a, b_1, b_2^0, b_2^1, b_3^0, b_3^1, c_1, c_2^0, c_2^1, c_3^0, c_3^1$ are all constants.

There are five solutions for the Chern polynomial equation if we set $u_3=0$.
\begin{center}
 	$(a,b_1,b_2^0,b_2^1,b_3^0,b_3^1,c_1^0,c_2^0,c_2^1,c_3^0,c_3^1)=(-2,0,0,0,0,0,2,4,4,8,8)$,
 	
 	$(a,b_1,b_2^0,b_2^1,b_3^0,b_3^1,c_1^0,c_2^0,c_2^1,c_3^0,c_3^1)=(-1,1,1,1,1,0,0,0,0,0)$,
 	
 	$(a,b_1,b_2^0,b_2^1,b_3^0,b_3^1,c_1^0,c_2^0,c_2^1,c_3^0,c_3^1)=(0,-1,1,1,-1,-1,1,0,0,0,0)$,
 	
 	$(a,b_1,b_2^0,b_2^1,b_3^0,b_3^1,c_1^0,c_2^0,c_2^1,c_3^0,c_3^1)=(0,-1,0,0,0,0,1,1,1,1,1)$,
 	
 	$(a,b_1,b_2^0,b_2^1,b_3^0,b_3^1,c_1^0,c_2^0,c_2^1,c_3^0,c_3^1)=(0,-2,2,3,0,-1,2,2,3,0,1)$.
 \end{center}

Using the similar arguments, we have following result.
\begin{proposition}
		The uniform vector bundle $E$ over $\mathbb{P}^4$ with splitting type $(3;1,3,3;u_1,u_1-1,u_1-2)$ is $$\mathcal{O}_{\mathbb{P}^{4}}(u_1)\oplus \mathcal{O}^{\oplus 3}_{\mathbb{P}^{4}}(u_1-1)\oplus \mathcal{O}^{\oplus 3}_{\mathbb{P}^{4}}(u_1-2),$$ $$T_{\mathbb{P}^4}(u_1-3)\oplus \mathcal{O}_{\mathbb{P}^4}(u_1)\oplus \mathcal{O}^{\oplus 2}_{\mathbb{P}^4}(u_1-1)$$ or $$\Omega^1_{\mathbb{P}^4}(u_1)\oplus \mathcal{O}_{\mathbb{P}^4}(u_1)\oplus\mathcal{O}^{\oplus 2}_{\mathbb{P}^4}(u_1-2).$$
\end{proposition}

 \vspace{1cm}

 \subsection{$k=3,r_1=r_3=3, r_2=1$}
We may assume
\[S_2(T,U,V)=T+a(U+V),\]
\[S_1(T,U,V)=T^3+b_1(U+V)T^2+(b_2^0U^2+b_2^1UV+b_2^0V^2)T+b_3^0(U^3+V^3)+b_3^1(U^2V+UV^2)\]
and \[S_3(T,U,V)=T^3+c_1(U+V)T^2+(c_2^0U^2+c_2^1UV+c_2^0V^2)T+c_3^0(U^3+V^3)+c_3^1(U^2V+UV^2),\]
 where $a, b_1, b_2^0, b_2^1, b_3^0, b_3^1, c_1, c_2^0, c_2^1, c_3^0, c_3^1$ are all constants.

There are six solutions for the Chern polynomial equation if we set $u_3=0$.
%
%
%
%
 \begin{center}
 	$(b_1,b_2^0,b_2^1,b_3^0,b_3^1,a,c_1,c_2^0,c_2^1,c_3^0,c_3^1)=(0,0,0,0,0,0,0,0,0,0,0)$,
 	
 	$(b_1,b_2^0,b_2^1,b_3^0,b_3^1,a,c_1,c_2^0,c_2^1,c_3^0,c_3^1)=(-1,1,1,-1,-1,1,0,0,0,0,0)$,
 	
 	$(b_1,b_2^0,b_2^1,b_3^0,b_3^1,a,c_1,c_2^0,c_2^1,c_3^0,c_3^1)=(0,0,0,0,0,-1,1,1,1,1,1)$,
 	
 	$(b_1,b_2^0,b_2^1,b_3^0,b_3^1,a,c_1,c_2^0,c_2^1,c_3^0,c_3^1)=(-2,4,4,-8,-8,0,2,0,0,0,0)$,
 	
 	$(b_1,b_2^0,b_2^1,b_3^0,b_3^1,a,c_1,c_2^0,c_2^1,c_3^0,c_3^1)=(-2,0,0,0,0,0,2,4,4,8,8)$,
 	
 	$(b_1,b_2^0,b_2^1,b_3^0,b_3^1,a,c_1,c_2^0,c_2^1,c_3^0,c_3^1)=(-4,8,12,0,-8,0,4,8,12,0,8)$.
 \end{center}

 Using the similar arguments, we have following results.
\begin{proposition}
	If $(c_1,c_2^0,c_2^1,c_3^0,c_3^1)=(0,0,0,0,0)$ or $(b_1,b_2^0,b_2^1,b_3^0,b_3^1)=(0,0,0,0,0)$, $E$ is $$\mathcal{O}^{\oplus 3}_{\mathbb{P}^{4}}(u_1)\oplus \mathcal{O}_{\mathbb{P}^{4}}(u_1-1)\oplus \mathcal{O}^{\oplus 3}_{\mathbb{P}^{4}}(u_1-2),$$ $$T_{\mathbb{P}^4}(u_1-3)\oplus \mathcal{O}^{\oplus 3}_{\mathbb{P}^4}(u_1)$$ or $$\Omega^1_{\mathbb{P}^4}(u_1+1)\oplus \mathcal{O}^{\oplus 3}_{\mathbb{P}^4}(u_1-2).$$
\end{proposition}
\begin{proposition}
	If $(b_1,b_2^0,b_2^1,b_3^0,b_3^1,a,c_1,c_2^0,c_2^1,c_3^0,c_3^1)=(-2,0,0,0,0,0,2,4,4,8,8)$ or
	
	$(b_1,b_2^0,b_2^1,b_3^0,b_3^1,a,c_1,c_2^0,c_2^1,c_3^0,c_3^1)=(-2,4,4,-8,-8,0,2,0,0,0,0)$,	
	there does not exist such a uniform vector bundle.
\end{proposition}

\begin{proposition}\label{exceptional-r313}
	Assume
	\[
	(b_1,b_2^0,b_2^1,b_3^0,b_3^1,a,c_1,c_2^0,c_2^1,c_3^0,c_3^1)
	=(-4,8,12,0,-8,0,4,8,12,0,8).
	\]
	There does not exist a uniform vector bundle $E$ giving this numerical solution.
\end{proposition}
\begin{proof}
After twisting, we may assume that the splitting type of $E$ is
$(3;3,1,3;2,1,0)$.  Put
\[
X=F(1,2,5),\qquad H=HN^1,\qquad N=HN^2,
\qquad M=p^*E/N,
\]
and set $K=M^\vee$.  We write
\[
\mathcal H=\mathcal H_{T_{\mathbb P^4}(-1)},\qquad
U=c_1(p^*\mathcal O_{\mathbb P^4}(-1)),\qquad
V=c_1(\mathcal H).
\]
The middle HN factor is trivial after normalization, so Lemma~\ref{trivial-lemma} gives
\begin{equation}\label{exceptional-middle-HN}
0\longrightarrow H\longrightarrow N\longrightarrow
p^*\mathcal O_{\mathbb P^4}(1)\longrightarrow0.
\end{equation}

For $x\in\mathbb P^4$, write $H_x=H|_{p^{-1}(x)}$.  The numerical solution gives
\[
c_1(H_x)=-4h,\qquad c_2(H_x)=8h^2,\qquad c_3(H_x)=0,
\]
and $H_x^\vee$ is globally generated.  A general section of $H_x^\vee$ has empty zero scheme, and hence
\begin{equation}\label{exceptional-fiber-quotient}
0\longrightarrow\mathcal O_{\mathbb P^3}
\longrightarrow H_x^\vee
\longrightarrow G_x\longrightarrow0,
\end{equation}
where $G_x$ is a globally generated rank-two bundle with $c_1(G_x)=4h$ and $c_2(G_x)=8h^2$.
It cannot split.  Moreover, $H^0(G_x^\vee)=0$: otherwise a nonzero morphism $G_x\to\mathcal O_{\mathbb P^3}$
would be surjective because $G_x$ is globally generated, and the resulting extension would split, contrary to
$c_2(G_x)=8h^2$.  By \cite[Proposition~1.6]{ggrk2}, a general section of $G_x$ gives
\begin{equation}\label{exceptional-curve}
0\longrightarrow\mathcal O_{\mathbb P^3}
\longrightarrow G_x
\longrightarrow\mathcal I_C(4)\longrightarrow0,
\end{equation}
where $C$ is either a smooth irreducible elliptic curve of degree eight or the disjoint union of two smooth elliptic
quartic curves.

Let $H'_x=G_x^\vee$.  Twisting \eqref{exceptional-curve} by $\mathcal O_{\mathbb P^3}(-4)$ and using Serre duality
gives
\[
h^0(H'_x)=h^3(H'_x)=0,
\qquad h^1(H'_x)=h^2(H'_x)=r-1,
\]
where $r=1$ or $2$ is the number of connected components of $C$.  Dualizing
\eqref{exceptional-fiber-quotient} gives
\begin{equation}\label{exceptional-fiber-extension}
0\longrightarrow H'_x\longrightarrow H_x\longrightarrow\mathcal O_{\mathbb P^3}\longrightarrow0.
\end{equation}
It follows that exactly one of the following occurs:
\begin{enumerate}
\item $h^0(H_x)=1$, $h^1(H_x)=h^2(H_x)=0$;
\item $h^0(H_x)=h^1(H_x)=0$, $h^2(H_x)=1$;
\item $h^0(H_x)=h^1(H_x)=h^2(H_x)=1$.
\end{enumerate}
In all cases $h^3(H_x)=0$ and $\chi(H_x)=1$.

Put
\[
J_H=\{x\in\mathbb P^4\mid h^0(H_x)\ne0\}.
\]
The set $J_H$ is closed by semicontinuity.  If $J_H=\mathbb P^4$, then $h^0(H_x)=1$ for every $x$.
Grauert's theorem and the same evaluation-map argument as in Proposition~\ref{exterior T} produce a line subbundle
$p^*\mathcal O_{\mathbb P^4}(j)\subset H$.  This is impossible because the Chern polynomial of $H$ has no factor
$T+jU$.  Hence $J_H\ne\mathbb P^4$.

Since $H\hookrightarrow p^*E$ and $p_*H$ vanishes away from $J_H$, we have $p_*H=0$.  On the dense open set
$\mathbb P^4\setminus J_H$ the second case above occurs, so $h^2(H_x)=1$.  Semicontinuity rules out the first case
at points of $J_H$; consequently $h^2(H_x)=1$ for every $x$.  Grauert's theorem now shows that
$R^2p_*H$ is a line bundle and that base change holds, while $R^3p_*H=0$.
In particular, the curve $C=C_1\sqcup C_2$ is the disjoint union of two elliptic quartics for every $x$.
The exact sequence used in \cite[Lemma~1.12]{ggrk2},
\[
0\longrightarrow\mathcal I_{C_2}\longrightarrow
2\mathcal I_{C_2}(2)\longrightarrow\mathcal I_C(4)\longrightarrow0,
\]
and the Koszul resolution of $C_2$ give
\[
h^0(\mathcal I_C(4))=4,\qquad h^1(\mathcal I_C(4))=1.
\]
Thus \eqref{exceptional-fiber-quotient} and \eqref{exceptional-curve} yield
\begin{equation}\label{exceptional-dual-cohomology}
h^0(H_x^\vee)=6,\qquad h^1(H_x^\vee)=1,
\qquad H^i(H_x^\vee)=0\quad(i\ge2).
\end{equation}

We shall also use the GRR class of $H$.  From $c_H(T)=S_1(T+2U,U,V)$ we obtain
\[
\begin{split}
c_1(H)&=-2U+4V,\\
c_2(H)&=4U^2-4UV+8V^2,\\
c_3(H)&=-8U^3-8UV^2.
\end{split}
\]
Using $p_*(V^3)=-1$, $p_*(V^4)=U$, and $p_*(V^m)=0$ for $m\ge5$, Grothendieck--Riemann--Roch gives
\[
\operatorname{ch}(p_!H)=1+\eta+\frac{\eta^2}{2}
+\frac{\eta^3}{6}+\frac{\eta^4}{24}
=\operatorname{ch}\bigl(\mathcal O_{\mathbb P^4}(1)\bigr).
\]
Here $\eta=c_1(\mathcal O_{\mathbb P^4}(1))=-U$.  Since $K_0(\mathbb P^4)$ is torsion-free,
\[
[p_!H]=[\mathcal O_{\mathbb P^4}(1)].
\]

Dualizing \eqref{exceptional-middle-HN} and using \eqref{exceptional-dual-cohomology}, we find
\[
h^0(N_x^\vee)=7,\qquad h^1(N_x^\vee)=1,
\qquad H^i(N_x^\vee)=0\quad(i\ge2).
\]
Consequently $P:=p_*N^\vee$ is a rank-seven vector bundle, base change holds, and the relative evaluation map is
surjective.  Let $R$ be its kernel:
\begin{equation}\label{exceptional-evaluation}
0\longrightarrow R\longrightarrow p^*P\longrightarrow N^\vee\longrightarrow0.
\end{equation}
For every $x$, this is the complete evaluation sequence of $N_x^\vee$.  Therefore
\begin{equation}\label{exceptional-R-cohomology}
h^0(R_x)=h^1(R_x)=0,\qquad h^2(R_x)=1,
\qquad h^3(R_x)=0,
\end{equation}
and
\[
c_1(R_x)=-4h,\qquad c_2(R_x)=8h^2,\qquad c_3(R_x)=0.
\]

We claim that
\begin{equation}\label{exceptional-quadratic-resolution}
0\longrightarrow\mathcal O_{\mathbb P^3}(-4)
\longrightarrow\mathcal O_{\mathbb P^3}(-2)^{\oplus4}
\longrightarrow R_x\longrightarrow0
\end{equation}
for every $x$.  Indeed, dualizing the evaluation sequence shows that $R_x^\vee$ is globally generated.
Moreover, $H^0(R_x)=H^1(R_x)=0$.  From \eqref{exceptional-curve} and
$G_x^\vee\simeq G_x(-4)$ one obtains
\[
H^1(G_x^\vee(-2))=H^1(G_x(-6))=0.
\]
Using \eqref{exceptional-fiber-extension} and \eqref{exceptional-middle-HN}, it follows that
$H^1(N_x(-2))=0$.  The dual evaluation sequence then gives
$H^0(R_x^\vee(-2))=0$.  We may therefore apply
\cite[Proposition~6.3 and Remark~6.4]{A-C-M} to $R_x^\vee$.
Since $c_3(R_x^\vee)=0$ and $h^2(R_x)=1$, the instanton case is excluded and case (ii) of that proposition gives
\eqref{exceptional-quadratic-resolution}.

Twisting \eqref{exceptional-quadratic-resolution} by $\mathcal O_{\mathbb P^3}(2)$ shows that
$h^0(R_x(2))=4$ and $H^i(R_x(2))=0$ for $i>0$.  Hence
\[
A:=p_*(R\otimes\mathcal H^{-2})
\]
is a rank-four bundle, and its relative evaluation map onto $R\otimes\mathcal H^{-2}$ is surjective.  Its kernel
is a line bundle whose restriction to every $p$-fiber is $\mathcal O_{\mathbb P^3}(-2)$.  Since
$\operatorname{Pic}(X)=p^*\operatorname{Pic}(\mathbb P^4)\oplus\mathbb Z[\mathcal H]$, the kernel is
$p^*D\otimes\mathcal H^2$ for a line bundle $D=\mathcal O_{\mathbb P^4}(l)$.  Tensoring the evaluation sequence by
$\mathcal H^2$ gives
\begin{equation}\label{exceptional-relative-resolution}
0\longrightarrow p^*D\otimes\mathcal H^4
\longrightarrow p^*A\otimes\mathcal H^2
\longrightarrow R\longrightarrow0.
\end{equation}

We next identify $R$.  Since $K=M^\vee\subset p^*E^\vee$, each $K_x^\vee$ is globally generated and $K_x$ has the same Chern classes as $H_x$.  Moreover, direct substitution shows that the Chern polynomial of $K$ has no factor $T+jU$.  Thus the preceding fiber and semicontinuity argument applies verbatim to $K$:
\[
p_*K=R^3p_*K=0,
\]
$L_K:=R^2p_*K$ is a line bundle, and $B_K:=R^1p_*K$ is a torsion sheaf.  Notice that $K$ has the same Chern
character as $H\otimes p^*\mathcal O_{\mathbb P^4}(-2)$.  Hence GRR and the projection formula give
\begin{equation}\label{exceptional-K-class}
[L_K]-[B_K]=[\mathcal O_{\mathbb P^4}(-1)]
\qquad\text{in }K_0(\mathbb P^4).
\end{equation}
Write $L_K=\mathcal O_{\mathbb P^4}(d)$.  Then
\[
c_1(B_K)=(d+1)\eta=-(d+1)U.
\]
Pushing forward
\[
0\longrightarrow K\longrightarrow p^*E^\vee\longrightarrow N^\vee\longrightarrow0
\]
gives
\[
0\longrightarrow E^\vee\longrightarrow P\longrightarrow B_K\longrightarrow0.
\]
Since $c_1(E^\vee)=7U$, we get
\begin{equation}\label{exceptional-c1-P}
c_1(P)=(6-d)U.
\end{equation}

We now compare Chern classes in \eqref{exceptional-evaluation} and
\eqref{exceptional-relative-resolution}.  First,
\[
c_1(N^\vee)=3U-4V,\qquad
c_2(N^\vee)=6U^2-8UV+8V^2.
\]
It follows from \eqref{exceptional-evaluation} and \eqref{exceptional-c1-P} that
\[
c_1(R)=(3-d)U+4V,
\]
and the $UV$- and $V^2$-coefficients of $c_2(R)$ are, respectively,
\begin{equation}\label{exceptional-c2-R-coefficients}
[c_2(R)]_{UV}=8-4d,\qquad [c_2(R)]_{V^2}=8.
\end{equation}
Write $c_i(A)=a_iU^i$.  Comparing degree one in \eqref{exceptional-relative-resolution}, and then comparing the
$UV$-coefficient in degree two, gives
\[
a_1+l=3-d,\qquad l=d-1,
\]
so that
\begin{equation}\label{exceptional-a1-l}
l=d-1,\qquad a_1=4-2d.
\end{equation}

Twisting \eqref{exceptional-relative-resolution} by $\mathcal H^{-2}$ gives
\[
0\longrightarrow\mathcal L:=p^*\mathcal O_{\mathbb P^4}(d-1)\otimes\mathcal H^2
\longrightarrow p^*A\longrightarrow R\otimes\mathcal H^{-2}\longrightarrow0.
\]
Put
\[
\lambda=c_1(\mathcal L)=-(d-1)U+2V.
\]
The quotient has rank three, so its fourth Chern class is zero:
\begin{equation}\label{exceptional-fourth-chern}
a_4U^4-\lambda a_3U^3+\lambda^2a_2U^2
-\lambda^3a_1U+\lambda^4=0
\quad\text{in }A^4(X).
\end{equation}
The only relation in degree four is
\[
R^4(U,V)=U^4+U^3V+U^2V^2+UV^3+V^4=0.
\]
The $V^4$-coefficient in the left side of \eqref{exceptional-fourth-chern} is $16$; hence that side, as a
polynomial, is $16R^4(U,V)$.  By \eqref{exceptional-a1-l}, its $UV^3$-coefficient is
\[
-8a_1-32(d-1)=-16d.
\]
Comparing this with the $UV^3$-coefficient of $16R^4(U,V)$ gives
\begin{equation}\label{exceptional-d-value}
d=-1.
\end{equation}

It now follows from \eqref{exceptional-K-class} that $[B_K]=0$.  A nonzero coherent sheaf has a nonzero Hilbert
polynomial, so $B_K=0$.  Consequently $E^\vee\simeq P$; under this identification the evaluation map in
\eqref{exceptional-evaluation} is the dual of $N\hookrightarrow p^*E$.  Thus
\begin{equation}\label{exceptional-R-is-K}
R\simeq K.
\end{equation}

Finally, \eqref{exceptional-a1-l} and \eqref{exceptional-d-value} give $l=-2$ and $a_1=6$.  We have
\[
c_1(K)=4U+4V,\qquad
c_2(K)=8U^2+12UV+8V^2.
\]
Using \eqref{exceptional-R-is-K} in \eqref{exceptional-relative-resolution}, comparison in degree two gives
\[
\begin{split}
c_2(p^*A\otimes\mathcal H^2)
&=c_2(K)+(2U+4V)(4U+4V)\\
&=16U^2+36UV+24V^2.
\end{split}
\]
On the other hand, twisting a rank-four bundle and using $c_1(A)=6U$ gives
\[
c_2(p^*A\otimes\mathcal H^2)=c_2(A)+36UV+24V^2.
\]
Therefore
\begin{equation}\label{exceptional-c2-A}
c_2(A)=16U^2.
\end{equation}

Restrict the sequence twisted by $\mathcal H^{-2}$ to the $q$-fiber corresponding to a line
$\ell\subset\mathbb P^4$.  Since $K|_{q^{-1}(\ell)}\simeq\mathcal O_\ell^{\oplus3}$ and
$\mathcal H|_{q^{-1}(\ell)}\simeq\mathcal O_\ell(1)$, we obtain
\[
0\longrightarrow\mathcal O_\ell\longrightarrow A|_\ell
\longrightarrow\mathcal O_\ell(-2)^{\oplus3}\longrightarrow0.
\]
This sequence splits because $H^1(\mathcal O_\ell(2))=0$.  Hence $A$ is uniform of splitting type
$(0,-2,-2,-2)$.  The classification of uniform rank-four bundles on $\mathbb P^4$
\cite{Sat,E-H-S} shows that $A$ must split: the tangent and cotangent cases have adjacent splitting degrees.  Thus
\[
A\simeq\mathcal O_{\mathbb P^4}\oplus
\mathcal O_{\mathbb P^4}(-2)^{\oplus3},
\]
which gives $c_2(A)=12U^2$, contradicting \eqref{exceptional-c2-A}. 
\end{proof}

\begin{proposition}
	Let $E$ be a uniform vector bundle over $\mathbb P^4$ with splitting type
	\[
	(3;3,1,3;u_1,u_1-1,u_1-2).
	\]
	Then $E$ is isomorphic to one of
	\[
	\mathcal O_{\mathbb P^4}^{\oplus3}(u_1)\oplus
	\mathcal O_{\mathbb P^4}(u_1-1)\oplus
	\mathcal O_{\mathbb P^4}^{\oplus3}(u_1-2),
	\]
	\[
	T_{\mathbb P^4}(u_1-3)\oplus\mathcal O_{\mathbb P^4}^{\oplus3}(u_1),
	\]
	or
	\[
	\Omega^1_{\mathbb P^4}(u_1+1)\oplus\mathcal O_{\mathbb P^4}^{\oplus3}(u_1-2).
	\]
\end{proposition}
\begin{proof}
After twisting, assume $u_1=2$ and $u_3=0$.  The Chern-polynomial equation has the six solutions listed above.
The solutions for which one of the rank-three HN factors is trivial give the three displayed homogeneous bundles.
The two solutions treated in the preceding nonexistence proposition do not occur, and the final solution
\[
(-4,8,12,0,-8,0,4,8,12,0,8)
\]
is excluded by Proposition~\ref{exceptional-r313}.
\end{proof}
\vspace{1cm}

Thus, all cases with $k=3$ have been classified.

When $k=4$, we divided this case into two parts.
\begin{enumerate}
\item $r_i=2,r_j=3$ for some $i, j$ such that $i<j$ without loss of generality.
\item $r_i=4$ for some $i$.
 \end{enumerate}

\subsection{$k=4, r_i=2,r_j=3$ for some $i, j$ and $i<j$}
In this case, we may assume $r_l=r_k=1$, $u_l> u_k$,
\begin{center}
	$S_{j}(T,U,V)=T^3+m_1(U+V)T^2+(m_2^0U^2+m_2^1UV+m_2^0V^2)T+m_3^0(U^3+V^3)+m_3^1(U^2V+UV^2)$,
	$S_{i}(T,U,V)=T^2+n_1(U+V)T+(n_2^0U^2+n_2^1UV+n_2^0V^2)$
\end{center}
and
\begin{center}
$S_{l}=T+a(U+V), ~S_{k}=T+b(U+V)$
\end{center}
where $a,b,m_1,m_2^0,m_2^1,m_3^0,m_3^1,n_1,n_2^0,n_2^1$ are all constants.

\subsubsection{$i=3,j=4$}
There are four solutions for the Chern polynomial equation if we set $u_4=0$.
\begin{center}
	$(a,b,n_1,n_2^0,n_2^1,m_1,m_2^0,m_2^1,m_3^0,m_3^1)=(-3,0,0,0,0,3,9,9,27,27)$,
	
	$(a,b,n_1,n_2^0,n_2^1,m_1,m_2^0,m_2^1,m_3^0,m_3^1)=(0,-2,0,0,0,2,4,4,8,8)$,
	
	$(a,b,n_1,n_2^0,n_2^1,m_1,m_2^0,m_2^1,m_3^0,m_3^1)=(0,0,-1,0,0,1,1,1,1,1)$,
	
	$(a,b,n_1,n_2^0,n_2^1,m_1,m_2^0,m_2^1,m_3^0,m_3^1)=(0,0,0,0,0,0,0,0,0,0)$.
\end{center}
\begin{proposition}
	The uniform vector bundle $E$ over $\mathbb{P}^4$ with splitting type $(4;1,1,2,3;u_1,u_1-1,u_1-2,u_1-3)$ is either $$\mathcal{O}_{\mathbb{P}^{4}}(u_1)\oplus \mathcal{O}_{\mathbb{P}^{4}}(u_1-1)\oplus \mathcal{O}^{\oplus 2}_{\mathbb{P}^{4}}(u_1-2)\oplus \mathcal{O}^{\oplus 3}_{\mathbb{P}^{4}}(u_1-3)$$ or $$T_{\mathbb{P}^4}(u_1-4)\oplus \mathcal{O}_{\mathbb{P}^4}(u_1)\oplus\mathcal{O}_{\mathbb{P}^4}(u_1-1)\oplus\mathcal{O}_{\mathbb{P}^4}(u_1-2).$$
\end{proposition}
\begin{proof}
If $(a,b,n_1,n_2^0,n_2^1,m_1,m_2^0,m_2^1,m_3^0,m_3^1)=(0,0,0,0,0,0,0,0,0,0)$,  then $E\cong\mathcal{O}_{\mathbb{P}^{4}}(u_1)\oplus \mathcal{O}_{\mathbb{P}^{4}}(u_1-1)\oplus \mathcal{O}^{\oplus 2}_{\mathbb{P}^{4}}(u_1-2)\oplus \mathcal{O}^{\oplus 3}_{\mathbb{P}^{4}}(u_1-2)$.

If $$(a,b,n_1,n_2^0,n_2^1,m_1,m_2^0,m_2^1,m_3^0,m_3^1)=(-3,0,0,0,0,3,9,9,27,27)$$
	or	
	$$(a,b,n_1,n_2^0,n_2^1,m_1,m_2^0,m_2^1,m_3^0,m_3^1)=(0,-2,0,0,0,2,4,4,8,8),$$ by the same argument in Proposition \ref{-2 case}, there does not exist such uniform vector bundles.

If  $(a,b,n_1,n_2^0,n_2^1,m_1,m_2^0,m_2^1,m_3^0,m_3^1)=(0,0,-1,0,0,1,1,1,1,1)$, then by the same argument in Proposition \ref{(-1,0,0)case},  $E\cong T_{\mathbb{P}^4}(u_1-4)\oplus \mathcal{O}_{\mathbb{P}^4}(u_1)\oplus\mathcal{O}_{\mathbb{P}^4}(u_1-1)\oplus\mathcal{O}_{\mathbb{P}^4}(u_1-2)$ .

\end{proof}

\subsubsection{$i=2,j=4$}

There are four solutions for the Chern polynomial equation if we set $u_4=0$.
\begin{center}
	$(a,n_1,n_2^0,n_2^1,b,m_1,m_2^0,m_2^1,m_3^0,m_3^1)=(-3,0,0,0,0,3,9,9,27,27)$,
	
	$(a,n_1,n_2^0,n_2^1,b,m_1,m_2^0,m_2^1,m_3^0,m_3^1)=(0,0,0,0,-1,1,1,1,1)$,
	
	$(a,n_1,n_2^0,n_2^1,b,m_1,m_2^0,m_2^1,m_3^0,m_3^1)=(0,-2,0,0,0,2,4,4,8,8)$,
	
	$(a,n_1,n_2^0,n_2^1,b,m_1,m_2^0,m_2^1,m_3^0,m_3^1)=(0,0,0,0,0,0,0,0,0,0)$.
\end{center}

Using the similar arguments, we have following result.
\begin{proposition}
The uniform vector bundle $E$ over $\mathbb{P}^4$ with splitting type $(4;1,2,1,3;u_1,u_1-1,u_1-2,u_1-3)$ is either $$\mathcal{O}_{\mathbb{P}^{4}}(u_1)\oplus \mathcal{O}^{\oplus 2}_{\mathbb{P}^{4}}(u_1-1)\oplus \mathcal{O}_{\mathbb{P}^{4}}(u_1-2)\oplus \mathcal{O}^{\oplus 3}_{\mathbb{P}^{4}}(u_1-3)$$ or $$T_{\mathbb{P}^4}(u_1-4)\oplus \mathcal{O}_{\mathbb{P}^4}(u_1)\oplus\mathcal{O}^{\oplus 2}_{\mathbb{P}^4}(u_1-1).$$
	\end{proposition}
%
%

\subsubsection{$i=1,j=4$}
There are four solutions for the Chern polynomial equation if we set $u_4=0$.
\begin{center}
	$(n_1,n_2^0,n_2^1,a,b,m_1,m_2^0,m_2^1,m_3^0,m_3^1)=(0,0,0,-2,0,2,4,4,8,8)$,
	
	$(n_1,n_2^0,n_2^1,a,b,m_1,m_2^0,m_2^1,m_3^0,m_3^1)=(0,0,0,0,-1,1,1,1,1,1)$,
	
	$(n_1,n_2^0,n_2^1,a,b,m_1,m_2^0,m_2^1,m_3^0,m_3^1)=(-3,0,0,0,0,3,9,9,27,27)$,
	
	$(n_1,n_2^0,n_2^1,a,b,m_1,m_2^0,m_2^1,m_3^0,m_3^1)=(0,0,0,0,0,0,0,0,0,0)$.
\end{center}
Using the similar arguments, we have following result.
\begin{proposition}
The uniform vector bundle $E$ over $\mathbb{P}^4$ with splitting type $(4;2,1,1,3;u_1,u_1-1,u_1-2,u_1-3)$ is either $$\mathcal{O}^{\oplus 2}_{\mathbb{P}^{4}}(u_1)\oplus \mathcal{O}_{\mathbb{P}^{4}}(u_1-1)\oplus \mathcal{O}_{\mathbb{P}^{4}}(u_1-2)\oplus \mathcal{O}^{\oplus 3}_{\mathbb{P}^{4}}(u_1-3)$$ or $$T_{\mathbb{P}^4}(u_1-4)\oplus \mathcal{O}^{\oplus 2}_{\mathbb{P}^4}(u_1)\oplus\mathcal{O}_{\mathbb{P}^4}(u_1-1).$$
	\end{proposition}
%
%
%
%

\subsubsection{$i=2,j=3$}
There are four solutions for the Chern polynomial equation if we set $u_4=0$.
\begin{center}
	$(a,n_1,n_2^0,n_2^1,m_1,m_2^0,m_2^1,m_3^0,m_3^1,b)=(-2,0,0,0,2,4,4,8,8,0)$,
	
	$(a,n_1,n_2^0,n_2^1,m_1,m_2^0,m_2^1,m_3^0,m_3^1,b)=(0,-1,0,0,1,1,1,1,1,0)$,
	
	$(a,n_1,n_2^0,n_2^1,m_1,m_2^0,m_2^1,m_3^0,m_3^1,b)=(0,0,0,0,0,0,0,0,0,0)$,
	
	$(a,n_1,n_2^0,n_2^1,m_1,m_2^0,m_2^1,m_3^0,m_3^1,b)=(0,0,0,0,-1,1,1,-1,-1,1)$.
\end{center}

\begin{proposition}
The uniform vector bundle $E$ over $\mathbb{P}^4$ with splitting type $(4;1,2,3,1;u_1,u_1-1,u_1-2,u_1-3)$ is $$\mathcal{O}_{\mathbb{P}^{4}}(u_1)\oplus \mathcal{O}^{\oplus 2}_{\mathbb{P}^{4}}(u_1-1)\oplus \mathcal{O}^{\oplus 3}_{\mathbb{P}^{4}}(u_1-2)\oplus \mathcal{O}_{\mathbb{P}^{4}}(u_1-3),$$ $$T_{\mathbb{P}^4}(u_1-3)\oplus \mathcal{O}_{\mathbb{P}^4}(u_1)\oplus\mathcal{O}_{\mathbb{P}^4}(u_1-1)\oplus\mathcal{O}_{\mathbb{P}^4}(u_1-3),$$ or$$\Omega^1_{\mathbb{P}^4}(u_1-1)\oplus \mathcal{O}_{\mathbb{P}^4}(u_1)\oplus\mathcal{O}^{\oplus 2}_{\mathbb{P}^4}(u_1-1).$$
	\end{proposition}
%
%
%

\subsubsection{$i=1,j=2$}

There are four solutions for the Chern polynomial equation if we set $u_4=0$.
\begin{center}
	$(n_1,n_2^0,n_2^1,m_1,m_2^0,m_2^1,m_3^0,m_3^1,a,b)=(-1,0,0,1,1,1,1,1,0,0)$,
	
	$(n_1,n_2^0,n_2^1,m_1,m_2^0,m_2^1,m_3^0,m_3^1,a,b)=(0,0,0,0,0,0,0,0,0,0)$,
	
	$(n_1,n_2^0,n_2^1,m_1,m_2^0,m_2^1,m_3^0,m_3^1,a,b)=(0,0,0,-1,1,1,-1,-1,1,0)$,
	
	$(n_1,n_2^0,n_2^1,m_1,m_2^0,m_2^1,m_3^0,m_3^1,a,b)=(0,0,0,-2,4,4,-8,-8,0,2)$.
\end{center}
Using the similar arguments, we have following result.
\begin{proposition}
The uniform vector bundle $E$ over $\mathbb{P}^4$ with splitting type $(4;2,3,1,1;u_1,u_1-1,u_1-2,u_1-3)$ is  $$\mathcal{O}^{\oplus 2}_{\mathbb{P}^{4}}(u_1)\oplus \mathcal{O}^{\oplus 3}_{\mathbb{P}^{4}}(u_1-1)\oplus \mathcal{O}_{\mathbb{P}^{4}}(u_1-2)\oplus \mathcal{O}_{\mathbb{P}^{4}}(u_1-3),$$  $$T_{\mathbb{P}^4}(u_1-2)\oplus \mathcal{O}_{\mathbb{P}^4}(u_1)\oplus\mathcal{O}_{\mathbb{P}^4}(u_1-2)\oplus\mathcal{O}_{\mathbb{P}^4}(u_1-3)$$ or $$\Omega^1_{\mathbb{P}^4}(u_1)\oplus \mathcal{O}^{\oplus 2}_{\mathbb{P}^4}(u_1)\oplus\mathcal{O}_{\mathbb{P}^4}(u_1-3).$$
	\end{proposition}
%
%
%

\subsubsection{$i=1,j=3$}
There are four solutions for the Chern polynomial equation if we set $u_4=0$.
\begin{center}
	$(n_1,n_2^0,n_2^1,a,m_1,m_2^0,m_2^1,m_3^0,m_3^1,b)=(0,0,0,-1,1,1,1,1,1,0)$,
	
	$(n_1,n_2^0,n_2^1,a,m_1,m_2^0,m_2^1,m_3^0,m_3^1,b)=(-2,0,0,0,2,4,4,8,8,0)$,
	
	$(n_1,n_2^0,n_2^1,a,m_1,m_2^0,m_2^1,m_3^0,m_3^1,b)=(0,0,0,0,0,0,0,0,0,0)$,
	
	$(n_1,n_2^0,n_2^1,a,m_1,m_2^0,m_2^1,m_3^0,m_3^1,b)=(0,0,0,0,-1,1,1,-1,-1,0)$,
\end{center}
Using the similar arguments, we have following result.

\begin{proposition}
The uniform vector bundle $E$ over $\mathbb{P}^4$ with splitting type $(4;2,1,3,1;u_1,u_1-1,u_1-2,u_1-3)$ is  $$\mathcal{O}^{\oplus 2}_{\mathbb{P}^{4}}(u_1)\oplus \mathcal{O}_{\mathbb{P}^{4}}(u_1-1)\oplus \mathcal{O}^{\oplus 3}_{\mathbb{P}^{4}}(u_1-2)\oplus \mathcal{O}_{\mathbb{P}^{4}}(u_1-3),$$  $$T_{\mathbb{P}^4}(u_1-2)\oplus \mathcal{O}^{\oplus 2}_{\mathbb{P}^4}(u_1)\oplus\mathcal{O}_{\mathbb{P}^4}(u_1-1)$$ or $$\Omega^1_{\mathbb{P}^4}(u_1-1)\oplus \mathcal{O}^{\oplus 2}_{\mathbb{P}^4}(u_1)\oplus\mathcal{O}_{\mathbb{P}^4}(u_1-1).$$
	\end{proposition}

%
%
%
\vspace{1cm}

 \subsection{$k=4,r_i=4$  for some $i$}
 In this case, we may assume $u_j>u_k>u_l$,
\begin{multline*}S_{i}(T,U,V)=T^4+n_1(U+V)T^3+(n_2^0U^2+n_2^1UV+n_2^0V^2)T^2+\\(n_3^0U^3+n_3^1U^2V+n_3^1UV^2+n_3^0V^3)T+n_4^0(U^4+V^4)+n_4^1(U^3V+UV^3)+n_4^2U^2V^2,
\end{multline*}
\[S_{j}(T,U,V)=T+a(U+V),\]
\[S_{k}(T,U,V)=T+b(U+V),\]
\[S_{l}(T,U,V)=T+c(U+V),\]
 where $a,b,c,n_1,n_2^0,n_2^1,n_3^0,n_3^1,n_4^0,n_4^1,n_4^2$ are all constants.

 \subsubsection{$i=4$}
There are four solutions for the Chern polynomial equation if we set $u_4=0$.
 \begin{center}
 	$(a,b,c,n_1,n_2^0,n_2^1,n_3^0,n_3^1)=(-3,0,0,3,9,9,27,27)$,
 	
 	$(a,b,c,n_1,n_2^0,n_2^1,n_3^0,n_3^1)=(0,-2,0,2,4,4,8,8)$,
 	
 	$(a,b,c,n_1,n_2^0,n_2^1,n_3^0,n_3^1)=(0,0,-1,1,1,1,1,1)$,
 	
 	$(a,b,c,n_1,n_2^0,n_2^1,n_3^0,n_3^1)=(0,0,0,0,0,0,0,0)$.	
 	 \end{center}
 Using the similar arguments, we have following result.
 \begin{proposition}
The uniform vector bundle $E$ over $\mathbb{P}^4$ with splitting type $(4;1,1,1,4;u_1,u_1-1,u_1-2,u_1-3)$ is either $$\mathcal{O}_{\mathbb{P}^{4}}(u_1)\oplus \mathcal{O}_{\mathbb{P}^{4}}(u_1-1)\oplus \mathcal{O}_{\mathbb{P}^{4}}(u_1-2)\oplus \mathcal{O}^{\oplus 4}_{\mathbb{P}^{4}}(u_1-3)$$ or  $$T_{\mathbb{P}^4}(u_1-4)\oplus \mathcal{O}_{\mathbb{P}^4}(u_1)\oplus\mathcal{O}_{\mathbb{P}^4}(u_1-1)\oplus\mathcal{O}_{\mathbb{P}^4}(u_1-3).$$
	\end{proposition}

%
%
%

 \subsubsection{$i=3$}
There are four solutions for the Chern polynomial equation if we set $u_4=0$.
 \begin{center}
 	$(a,b,n_1,n_2^0,n_2^1,n_3^0,n_3^1,c)=(-2,0,2,4,4,8,8,0)$,
 	
 	$(a,b,n_1,n_2^0,n_2^1,n_3^0,n_3^1,c)=(0,-1,1,1,1,1,1,0)$,
 	
 	$(a,b,n_1,n_2^0,n_2^1,n_3^0,n_3^1,c)=(0,0,0,0,0,0,0,0)$,
 	
 	$(a,b,n_1,n_2^0,n_2^1,n_3^0,n_3^1,c)=(0,0,-1,1,1,-1,-1,1)$.
 \end{center}
  Using the similar arguments, we have following result.
  \begin{proposition}
The uniform vector bundle $E$ over $\mathbb{P}^4$ with splitting type $(4;1,1,4,1;u_1,u_1-1,u_1-2,u_1-3)$ is  $$\mathcal{O}_{\mathbb{P}^{4}}(u_1)\oplus \mathcal{O}_{\mathbb{P}^{4}}(u_1-1)\oplus \mathcal{O}^{\oplus 4}_{\mathbb{P}^{4}}(u_1-2)\oplus \mathcal{O}_{\mathbb{P}^{4}}(u_1-3),$$  $$T_{\mathbb{P}^4}(u_1-3)\oplus \mathcal{O}_{\mathbb{P}^4}(u_1)\oplus\mathcal{O}_{\mathbb{P}^4}(u_1-2)\oplus\mathcal{O}_{\mathbb{P}^4}(u_1-3)$$ or $$\Omega^1_{\mathbb{P}^4}(u_1-1)\oplus \mathcal{O}_{\mathbb{P}^4}(u_1)\oplus\mathcal{O}_{\mathbb{P}^4}(u_1-1)\oplus\mathcal{O}_{\mathbb{P}^4}(u_1-2).$$
	\end{proposition}

%
%
%

So we have finished classification of uniform $7$-bundles over $\mathbb{P}^4$ with $k=4$.

\vspace{1cm}

\subsection{$k=5,r_i=3$ for some $i$}
 We may assume $u_j>u_k>u_l>u_m$.
 \begin{center}
 	$S_{i}(T,U,V)=T^3+n_1(U+V)T^2+(n_2^0U^2+n_2^1UV+n_2^0V^2)T+(n_3^0U^3+n_3^1U^2V+n_3^1UV^2+n_3^0V^3)$, 	
 	$S_{j}(T,U,V)=T+a(U+V)$, 	
 	$S_{k}(T,U,V)=T+b(U+V)$, 	
 	$S_{l}(T,U,V)=T+c(U+V)$, 	
 	$S_{m}(T,U,V)=T+d(U+V)$.
 \end{center}
 where $a,b,c,d, n_1,n_2^0,n_2^1,n_3^0,n_3^1$ are all constants.
 \subsubsection{$i=5$}
There are five solutions for the Chern polynomial equation if we set $u_4=0$.
 \begin{center}
 	$(a,b,c,d,n_1,n_2^0,n_2^1,n_3^0,n_3^1)=(0,0,0,0,0,0,0,0,0)$,
 	
 	$(a,b,c,d,n_1,n_2^0,n_2^1,n_3^0,n_3^1)=(0,0,0,-1,1,1,1,1,1)$,
 	
 	$(a,b,c,d,n_1,n_2^0,n_2^1,n_3^0,n_3^1)=(0,0,-2,0,2,4,4,8,8)$,
 	
 	$(a,b,c,d,n_1,n_2^0,n_2^1,n_3^0,n_3^1)=(0,-3,0,0,3,9,9,27,27)$,
 	
 	$(a,b,c,d,n_1,n_2^0,n_2^1,n_3^0,n_3^1)=(-4,0,0,0,4,16,16,64,64)$.
 \end{center}
   Using the similar arguments, we have following result.
   \begin{proposition}
The uniform vector bundle $E$ over $\mathbb{P}^4$ with splitting type $(5;1,1,1,1,3;u_1,u_1-1,u_1-2,u_1-3,u_1-4)$ is either $$\mathcal{O}_{\mathbb{P}^{4}}(u_1)\oplus \mathcal{O}_{\mathbb{P}^{4}}(u_1-1)\oplus \mathcal{O}_{\mathbb{P}^{4}}(u_1-2)\oplus \mathcal{O}_{\mathbb{P}^{4}}(u_1-3)\oplus \mathcal{O}^{\oplus 3}_{\mathbb{P}^{4}}(u_1-4)$$ or  $$T_{\mathbb{P}^4}(u_1-5)\oplus \mathcal{O}_{\mathbb{P}^4}(u_1)\oplus\mathcal{O}_{\mathbb{P}^4}(u_1-1)\oplus\mathcal{O}_{\mathbb{P}^4}(u_1-2).$$
	\end{proposition}

%
%
%
%

 \subsubsection{$i=4$}
There are five solutions for the Chern polynomial equation if we set $u_4=0$.
 \begin{center}
 	$(a,b,c,n_1,n_2^0,n_2^1,n_3^0,n_3^1,d)=(0,0,0,-1,1,1,-1,-1,1)$,
 	
 	$(a,b,c,n_1,n_2^0,n_2^1,n_3^0,n_3^1,d)=(0,0,0,0,0,0,0,0,0)$,
 	
 	$(a,b,c,n_1,n_2^0,n_2^1,n_3^0,n_3^1,d)=(0,0,-1,1,1,1,1,1,0)$,
 	
 	$(a,b,c,n_1,n_2^0,n_2^1,n_3^0,n_3^1,d)=(0,-2,0,2,4,4,8,8,0)$,
 	
 	$(a,b,c,n_1,n_2^0,n_2^1,n_3^0,n_3^1,d)=(-3,0,0,3,9,9,27,27,0)$.
 \end{center}
  Using the similar arguments, we have following result.
  \begin{proposition}
The uniform vector bundle $E$ over $\mathbb{P}^4$ with splitting type $(5;1,1,1,3,1;u_1,u_1-1,u_1-2,u_1-3,u_1-4)$ is  $$\mathcal{O}_{\mathbb{P}^{4}}(u_1)\oplus \mathcal{O}_{\mathbb{P}^{4}}(u_1-1)\oplus \mathcal{O}_{\mathbb{P}^{4}}(u_1-2)\oplus \mathcal{O}^{\oplus 3}_{\mathbb{P}^{4}}(u_1-3)\oplus \mathcal{O}_{\mathbb{P}^{4}}(u_1-1),$$  $$T_{\mathbb{P}^4}(u_1-4)\oplus \mathcal{O}_{\mathbb{P}^4}(u_1)\oplus\mathcal{O}_{\mathbb{P}^4}(u_1-1)\oplus\mathcal{O}_{\mathbb{P}^4}(u_1-4)$$ or $$\Omega^1_{\mathbb{P}^4}(u_1-2)\oplus \mathcal{O}_{\mathbb{P}^4}(u_1)\oplus\mathcal{O}_{\mathbb{P}^4}(u_1-1)\oplus\mathcal{O}_{\mathbb{P}^4}(u_1-2).$$
	\end{proposition}

%
%
%
%

 \subsubsection{$i=3$}
There are five solutions for the Chern polynomial equation if we set $u_4=0$.
 \begin{center}
 	$(a,b,n_1,n_2^0,n_2^1,n_3^0,n_3^1,c,d)=(0,0,-2,4,4,-8,-8,0,2)$,
 	
 	$(a,b,n_1,n_2^0,n_2^1,n_3^0,n_3^1,c,d)=(0,0,-1,1,1,-1,-1,1,0)$,
 	
 	$(a,b,n_1,n_2^0,n_2^1,n_3^0,n_3^1,c,d)=(0,0,0,0,0,0,0,0)$,
 	
 	$(a,b,n_1,n_2^0,n_2^1,n_3^0,n_3^1,c,d)=(0,-1,1,1,1,1,1,0,0)$,
 	
 	$(a,b,n_1,n_2^0,n_2^1,n_3^0,n_3^1,c,d)=(-2,0,2,4,4,8,8,0,0)$.
 \end{center}
 Using the similar arguments, we have following result.
  \begin{proposition}
The uniform vector bundle $E$ over $\mathbb{P}^4$ with splitting type $(5;1,1,3,1,1;u_1,u_1-1,u_1-2,u_1-3,u_1-4)$ is  $$\mathcal{O}_{\mathbb{P}^{4}}(u_1)\oplus \mathcal{O}_{\mathbb{P}^{4}}(u_1-1)\oplus \mathcal{O}^{\oplus 3}_{\mathbb{P}^{4}}(u_1-2)\oplus \mathcal{O}_{\mathbb{P}^{4}}(u_1-3)\oplus \mathcal{O}_{\mathbb{P}^{4}}(u_1-1),$$  $$T_{\mathbb{P}^4}(u_1-3)\oplus \mathcal{O}_{\mathbb{P}^4}(u_1)\oplus\mathcal{O}_{\mathbb{P}^4}(u_1-3)\oplus\mathcal{O}_{\mathbb{P}^4}(u_1-4)$$ or $$\Omega^1_{\mathbb{P}^4}(u_1-1)\oplus \mathcal{O}_{\mathbb{P}^4}(u_1)\oplus\mathcal{O}_{\mathbb{P}^4}(u_1-1)\oplus\mathcal{O}_{\mathbb{P}^4}(u_1-4).$$
	\end{proposition}

%
%
%
%
So we have finished the classification of uniform $7$-bundle over $\mathbb{P}^4$ with $k=5$.

 \vspace{1cm}

 \subsection{$r_i\leq 2$ for all $i$}
 In this case, the solutions of the Chern polynomials can only be zero.
 So all $S_{i}(T,U,V)$ are $T^{r_i}$, by the same arguments, such a uniform vector bundle is the direct sum of line bundles.

 So the following theorems hold.
 \begin{theorem}\label{r7k6}
 	Any uniform $7$-bundle over $\mathbb{P}^4$ with $k\geq 6$ are direct sum of line bundles.
 \end{theorem}

 \begin{theorem}\label{rk7}
 	Every uniform $7$-bundle over $\mathbb{P}^4$ is homogeneous.
 \end{theorem}
%

	Finally, we have finished proving the following theorem.
\begin{theorem}
	Uniform $r$-bundles with $r<8$ on $\mathbb{P}^4$ are homogeneous.
\end{theorem}

 \section{Appendix}
All calculations mentioned in this paper were completed using Mathematica 11.3. However, given the large number of classification cases and the similarity of the corresponding programs, listing all codes here would significantly increase the length of the paper while adding little substantive value. Therefore, we only present the calculation programs for the cases of rank $6$ and $7$ with $k=2$, denoted as the $(3,3)$-case and $(3,4)$-case respectively. The remaining programs are available at \url{https://github.com/zhoyuhang/mathematica-code/tree/v1.0} or \url{https://math.ecnu.edu.cn/~rdu/math code.html}.

$(3,3)$-case:
\begin{lstlisting}
In[1]=Solve[{
m11 + n1 == 0,	
m21 + n21 + m11*n1 + m10*n1 == 0,
m22 + n20 + m11*n1 == 0,		
m31 + n31 + m10*n21 + m11*n20 + m21*n1 + m20*n1 == 0,
m32 + n31 + m10*n20 + m11*n21 + m22*n1 + m21*n1 == 0,
m33 + n30 + m11*n20 + m22*n1 == 0,

m31*n1 + m30*n1 + m21*n20 + m20*n21 + m10*n31 +
m11*n30 - (m33*n1 + m22*n20 + m11*n30) == 0,
m32*n1 + m31*n1 + m22*n20 + m21*n21 + m20*n20 + m10*n31 +
m11*n31 - (m33*n1 + m22*n20 + m11*n30) == 0,
m33*n1 + m32*n1 + m22*n21 + m21*n20 + m10*n30 +
m11*n31 - (m33*n1 + m22*n20 + m11*n30) == 0,
		
n31*m20 + n30*m21 + n20*m31 +
n21*m30 - (n30*m21 + n31*m22 + n21*m33 + n20*m32) == 0,
n31*m20 + n31*m21 + n30*m22 + n20*m32 + n21*m31 +
n20*m30 - (n30*m21 + n31*m22 + n21*m33 + n20 m32) == 0,
n30*m20 + n31*m21 + n31*m22 + n20*m33 + n21*m32 +
n20*m31 - (n30*m21 + n31*m22 + n21*m33 + n20*m32) == 0,
		
n30*m32 + n31*m31 + n31*m30 == (n31*m33 + n31*m32 + n30*m31),
n30*m33 + n31*m32 + n31*m31 +
n30*m30 == (n31*m33 + n31*m32 + n30*m31),
		
m10 - (t1 + 3) == 0,
m11 - t1 == 0,
m20 - (3 + 2*t1 + t20) == 0,
m21 - (t21 + 2*t1) == 0,
m22 - t20 == 0,
m30 - (1 + t1 + t20 + t30) == 0,
m31 - (t1 + t21 + t31) == 0,
m32 - (t20 + t31) == 0,
m33 - t30 == 0},
{n1, n20, n21, n30, n31, m10, m11, m20, m21, m22, m30, m31, m32, m33,
		t1, t20, t21, t30, t31}, Integers]
\end{lstlisting}

\begin{lstlisting}
Out[1]={{n1 -> 0, n20 -> 0, n21 -> 0, n30 -> 0, n31 -> 0, m10 -> 3,
m11 -> 0, m20 -> 3, m21 -> 0, m22 -> 0, m30 -> 1, m31 -> 0, m32 -> 0,
m33 -> 0, t1 -> 0, t20 -> 0, t21 -> 0, t30 -> 0, t31 -> 0},
	
{n1 -> 1, n20 -> 0, n21 -> 0, n30 -> 0, n31 -> 0,
m10 -> 2, m11 -> -1, m20 -> 2, m21 -> -1, m22 -> 1, m30 -> 0,
m31 -> -1, m32 -> 0, m33 -> -1, t1 -> -1, t20 -> 1, t21 -> 1,
t30 -> -1, t31 -> -1},
	
{n1 -> 1, n20 -> 1, n21 -> 1, n30 -> 1,
n31 -> 1, m10 -> 2, m11 -> -1, m20 -> 1, m21 -> -2, m22 -> 0,
m30 -> 0, m31 -> -1, m32 -> 0, m33 -> 0, t1 -> -1, t20 -> 0,
t21 -> 0, t30 -> 0, t31 -> 0},
	
{n1 -> 2, n20 -> 2, n21 -> 3,
n30 -> 0, n31 -> 1, m10 -> 1, m11 -> -2, m20 -> 1, m21 -> -1,
m22 -> 2, m30 -> 1, m31 -> 0, m32 -> 1, m33 -> 0, t1 -> -2,
t20 -> 2, t21 -> 3, t30 -> 0, t31 -> -1}}
\end{lstlisting}

$(3,4)$-case:
\begin{lstlisting}
In[1]=Solve[{
n10 - n11 == 0,
n20 - n22 == 0,
n30 - n33 == 0,
n31 - n32 == 0,
n40 == n44,
n41 == n43,
m10 - (t10 + 3) == 0,
m11 - t10 == 0,
m20 - (3 + 2*t10 + t20) == 0,
m21 - (t21 + 2*t10) == 0,
m22 - t20 == 0,
m30 - (1 + t10 + t20 + t30) == 0,
m31 - (t10 + t21 + t31) == 0,
m32 - (t20 + t31) == 0,
m33 - t30 == 0,
		
m11 + n11 == 0,
	
m21 + n21 + m11*n10 + m10*n11 == 0,
m22 + n22 + m11*n11 == 0,

m31 + n31 + m10*n21 + m11*n20 + m21*n10 + m20*n11 == 0,
m32 + n32 + m10*n22 + m11*n21 + m22*n10 + m21*n11 == 0,
m33 + n33 + m11*n22 + m22*n11 == 0,
		
m31*n10 + m30*n11 + m21*n20 + m20*n21 + m10*n31 + m11*n30 +
n41 - (m33*n11 + m22*n22 + m11*n33 + n44) == 0,
m32*n10 + m31*n11 + m22*n20 + m21*n21 + m20*n22 + m10*n32 +
m11*n31 + n42 - (m33*n11 + m22*n22 + m11*n33 + n44) == 0,
m33*n10 + m32*n11 + m22*n21 + m21*n22 + m10*n33 + m11*n32 +
n43 - (m33*n11 + m22*n22 + m11*n33 + n44) == 0,

n41*m10 + n40*m11 + n31*m20 + n30*m21 + n20*m31 +
n21*m30 - (n44*m10 + n43*m11 + n33*m21 + n32*m22 + n21*m33 +
n22*m32) == 0,
n42*m10 + n41*m11 + n32*m20 + n31*m21 + n30*m22 + n20*m32 +
n21*m31 + n22*m30 -
(n44*m10 + n43*m11 + n33*m21 + n32*m22 + n21*m33 + n22*m32) == 0,
n43*m10 + n42*m11 + n33*m20 + n32*m21 + n31*m22 + n20*m33 +
n21*m32 +n22*m31 -
(n44*m10 + n43*m11 + n33*m21 + n32*m22 + n21*m33 + n22*m32) == 0,
		
n42*m20 + n41*m21 + n40*m22 + n30*m32 + n31*m31 +
n32*m30 - (n44*m20 + n43*m21 + n42*m22 + n31*m33 + n32*m32 +
n33*m31) == 0,
n43*m20 + n42*m21 + n41*m22 + n30*m33 + n31*m32 + n32*m31 +
n33*m30 - (n44*m20 + n43*m21 + n42*m22 + n31*m33 + n32*m32 +
n33*m31) == 0,
		
n40*m33 + n41*m32 + n42*m31 + n43*m30 -
(n41*m33 + n42*m32 + n43*m31 + n44*m30) == 0 },
	
{n10, n11, n20, n21, n22, n30, n31, n32, n33, n40, n41, n42, n43,
n44, m10, m11, m20, m21, m22, m30, m31, m32, m33, t10, t20, t21,
t30, t31}]
\end{lstlisting}

\begin{lstlisting}
Out[1]={{n10 -> 2, n11 -> 2, n20 -> 2, n21 -> 3, n22 -> 2, n30 -> 0,
n31 -> 1, n32 -> 1, n33 -> 0, n40 -> n43, n41 -> n43, n42 -> n43,
n44 -> n43, m10 -> 1, m11 -> -2, m20 -> 1, m21 -> -1, m22 -> 2,
m30 -> 1, m31 -> 0, m32 -> 1, m33 -> 0, t10 -> -2, t20 -> 2,
t21 -> 3, t30 -> 0, t31 -> -1},
	
{n10 -> 1, n11 -> 1, n20 -> 1,
n21 -> 1, n22 -> 1, n30 -> 1, n31 -> 1, n32 -> 1, n33 -> 1,
n40 -> n43, n41 -> n43, n42 -> n43, n44 -> n43, m10 -> 2, m11 -> -1,
m20 -> 1, m21 -> -2, m22 -> 0, m30 -> 0, m31 -> -1, m32 -> 0,
m33 -> 0, t10 -> -1, t20 -> 0, t21 -> 0, t30 -> 0,
t31 -> 0},
	
{n10 -> 1, n11 -> 1, n20 -> 0, n21 -> 0, n22 -> 0,
n30 -> 0, n31 -> 0, n32 -> 0, n33 -> 0, n40 -> n43, n41 -> n43,
n42 -> n43, n44 -> n43, m10 -> 2, m11 -> -1, m20 -> 2, m21 -> -1,
m22 -> 1, m30 -> 0, m31 -> -1, m32 -> 0, m33 -> -1, t10 -> -1,
t20 -> 1, t21 -> 1, t30 -> -1, t31 -> -1}, {n10 -> 0, n11 -> 0,
n20 -> 0, n21 -> 0, n22 -> 0, n30 -> 0, n31 -> 0, n32 -> 0,
n33 -> 0, n40 -> n43, n41 -> n43, n42 -> n43, n44 -> n43, m10 -> 3,
m11 -> 0, m20 -> 3, m21 -> 0, m22 -> 0, m30 -> 1, m31 -> 0,
m32 -> 0, m33 -> 0, t10 -> 0, t20 -> 0, t21 -> 0, t30 -> 0, t31 -> 0}}
\end{lstlisting}

\section*{Acknowledgements}
Both authors would like to show our great appreciation to the anonymous reviewer for providing many useful suggestions in the original version.
 \bibliographystyle{amsalpha}
 \bibliography{ref-P4R67}

@article {B-E1982,
	AUTHOR = {Ballico, E. and Ellia, P.},
	TITLE = {Fibr\'es homog\`enes sur $\mathbb{P}^n$},
	JOURNAL = {C. R. Acad. Sci. Paris S\'er. I Math.},
	VOLUME = {294},
	YEAR = {1982},
	NUMBER = {12},
	PAGES = {403--406}
	}

@article{Ven,
	title={On uniform vector bundles},
	author={A. {Van de Ven}},
	journal={Math. Ann.},
	volume={195},
    pages={245--248},
	year={1972}
}

@article{Sat,
	title={Uniform vector bundles on a projective space},
	author={E. Sato},
	journal={J. Math. Soc. Japan},
	volume={28},
	number={1},
	pages={123--132},
	year={1976}
}

@article{Ele,
	title={Les fibr\'es uniformes de rang 3 sur $\mathbb{P}^2(\mathbb{C})$ sont homog\'enes},
	author={G. Elencwajg},
	journal={Math.Ann},
	volume={231},
	pages={217--227},
	year={1978}
}

@article{E-H-S,
	title={Les fibr\'es uniformes de rang n sur $\mathbb{P}^n(\mathbb{C})$ sont ceux qu'on croit},
	author={G. Elencwajg and A. Hirschowitz and M. Schneider},
	journal={Progr. Math.},
	volume={7},
	pages={37--63},
	year={1980},
    publisher={Birkh$\ddot{a}$user, Boston, Mass.}
}

@article{Ell,
	title={Sur les fibr{\'e}s uniformes de rang $(n+1)$ sur $\mathbb{P}^{n} $},
	author={P. Ellia},
	journal={M{\'e}m. Soc. Math. France (N.S.)},
	volume={7},
	pages={1--60},
	year={1982}
}

@article{Bal,
	title={Uniform vector bundles of rank $(n+1)$ on $\mathbb{P}^{n}$},
	author={E. Ballico},
	journal={Tsukuba J. Math.},
	volume={7},
	number={2},
	pages={215--226},
	year={1983}
}

@article{Guy,
	title={Caract{\'e}risation par l'uniformit{\'e} des fibr{\'e}s universels sur la {G}rassmanienne},
	author={M. Guyot},
	journal={Math. Ann.},
	volume={270},
	number={1},
	pages={47--62},
	year={1985}
}

@book{O-S-S,
	title={Vector bundles on complex projective spaces},
	author={C. Okonek and M. Schneider and H. Spindler},
    year={2011},
    publisher={Birkh$\ddot{a}$user/Springer Basel AG, Basel},
    note={viii+239 pp.}
}

@article {S-J-U,
    AUTHOR = {Sierra, J. and Ugaglia, L.},
     TITLE = {On globally generated vector bundles on projective spaces},
   JOURNAL = {J. Pure Appl. Algebra},
     VOLUME = {213},
      YEAR = {2009},
    NUMBER = {11},
     PAGES = {2141--2146}
}

@article {EP-MP,
    AUTHOR = {Ellia, Ph. and Menegatti, P.},
     TITLE = {Spaces of matrices of constant rank and uniform vector
              bundles},
   JOURNAL = {Linear Algebra Appl.},
    VOLUME = {507},
      YEAR = {2016},
     PAGES = {474--485}
}

@article {B-E-P3,
    AUTHOR = {Ballico, E. and Ellia, P.},
     TITLE = {Fibr\'es uniformes de rang {$5$}\ sur {${\bf P}\sp{3}$}},
   JOURNAL = {Bull. Soc. Math. France},
    VOLUME = {111},
      YEAR = {1983},
    NUMBER = {1},
     PAGES = {59--87}
}

@article {Guyot,
    AUTHOR = {Guyot, M.},
     TITLE = {Caract\'erisation par l'uniformit\'e des fibr\'es universels
              sur la grassmanienne},
   JOURNAL = {Math. Ann.},
   VOLUME = {270},
      YEAR = {1985},
    NUMBER = {1},
     PAGES = {47--62}
}

@article {A-C-M,
    AUTHOR = {Anghel, C. and Coand\u{a}, I. and Manolache, N.},
     TITLE = {Globally generated vector bundles with small {$c_1$} on
              projective spaces},
   JOURNAL = {Mem. Amer. Math. Soc.},
   VOLUME = {253},
      YEAR = {2018},
    NUMBER = {1209},
     NOTE = {v+107}
}

@article{ggrk2,
  title={Rank two globally generated vector bundles with $c_1\leq 5$},
  author={L. Chiodera and P. Ellia},
  journal={Rend. Istit. Mat. Univ. Trieste},
  volume={44},
	pages={413--422},
  year={2012}
}

@article{Ele2,
	title={Des fibr{\'e}s uniformes non homog{\`e}nes},
	author={Elencwajg, G.},
	journal={Mathematische Annalen},
	volume={239},
	pages={185--192},
	year={1979}
}

@article{Dre,
	title={Exemples de fibr{\'e}s uniformes non homog{\`e}nes sur $\mathbb{P}_n$},
	author={Dr{\'e}zet, J.},
	journal={Comptes Rendus de l'Acad{\'e}mie des Sciences},
	volume={291},
	number={2},
	pages={125--128},
	year={1980}
}

\end{document}